\documentclass[12pt]{article}
\usepackage{amsmath,amssymb,amstext,dsfont,fancyvrb,float,fontenc,graphicx,caption,subcaption,theorem,hyperref}
\allowdisplaybreaks

\usepackage[utf8]{inputenc}

\usepackage{tikz,everypage}

\usepackage[letterpaper]{geometry}
\ifx\volno\undefined\def\volno{0}\fi
\ifx\volyear\undefined\def\volyear{2017}\fi
\ifx\pagno\undefined\def\pagno{000--000}\fi

\newfont{\footsc}{cmcsc10 at 8truept}
\newfont{\footbf}{cmbx10 at 8truept}
\newfont{\footrm}{cmr10 at 10truept}

\usepackage{fancyhdr}
\usepackage{relsize}
\usepackage{sectsty}
\allsectionsfont{\larger[-1]} 

\renewcommand\paragraph{\@startsection{paragraph}{4}{\z@}
                                    {2ex \@plus.5ex \@minus.2ex}
                                    {-1em}
                                    {\normalfont\normalsize\bfseries}}

\renewcommand\subparagraph{\@startsection{subparagraph}{5}{\parindent}
                                       {2ex \@plus.5ex \@minus .2ex}
                                       {-1em}
                                      {\normalfont\normalsize\bfseries}}

\usepackage[strict]{changepage}
\def\abstractname{Abstract -}   
\def\abstract{\begin{adjustwidth}{1cm}{1cm} \par    \footnotesize \noindent {\bf \abstractname} 
\def\endabstract{ \end{adjustwidth} \smallskip }}

{\theorembodyfont{\itshape}\newtheorem{theorem}{Theorem}[section]}
{\theorembodyfont{\itshape}}
{\theorembodyfont{\itshape}\newtheorem{definition}[theorem]{Definition}}
{\theorembodyfont{\itshape}\newtheorem{lemma}[theorem]{Lemma}}
{\theorembodyfont{\itshape}}
{\theorembodyfont{\rm}}
{\theorembodyfont{\rm}}
{\theorembodyfont{\rm}}
{\theorembodyfont{\rm }}

\title{\Large\bf Structure Constant Formulas for the Universal Enveloping Algebras of the Nilpotent Lie Algebras of Dimension Five and Less}
\author{\sc S. Chamberlin and E. Taylor }

\begin{document}
\setcounter{page}{1}
\maketitle
\thispagestyle{fancy}

\vskip 1.5em

\begin{abstract}
Libor \v{S}nobl and Pavel Winternitz classified all of the Lie algebras of dimension six and smaller. Using this classification, we formulated and proved structure constant formulas for the universal enveloping algebras of the nilpotent Lie algebras of dimension five and less.
\end{abstract}
 

\begin{MSC}
16S30; 17B05; 17B30; 17B35

\end{MSC}

\section{Introduction} 

In 2014, Libor \v{S}nobl and Pavel Winternitz classified all of the Lie algebras of dimension six and smaller, see \cite{SW}. For background on Lie algebras, see \cite{Carter}. Given a Lie algebra $L$, the  Poincaré-Birkhoff-Witt (PBW) Theorem implies that every element of the universal enveloping algebra of $L$, $\mathbf{U}(L)$, can be written uniquely as a linear combination of products of basis elements of $L$ in any fixed order. The coefficients of the resulting linear combination are referred to as the structure constants. 

To get the structure constants, we make use of straightening identities, or formulas for reordering noncommutative products within the universal enveloping algebra of a Lie algebra. Because of the noncommutative nature of Lie algebras, straightening identities are well used and well studied, see \cite{CN,CPR,CRSS,Kostant}. 

Structure constants for particular Lie algebras have been studied before. In 2020, Gourley and Kennedy derived a recursive formula for the structure constants of $\mathbf{U}(\mathfrak{sl}_2)$, see \cite{GK}. Then in 2022, using Kostant's straightening identity, see \cite{Kostant}, Chamberlin and Fernelius provided a closed formula for the structure constants of $\mathbf U ( \mathfrak{sl}_2 )$, see \cite{CF}. In a slightly different manner, Bremner, Hentzel, Peresi, and Usefi determined the structure constants for the universal enveloping algebra of the non-Lie four dimensional Malcev algebra, see \cite{BHPU}.

\section{Preliminaries}

Let $\mathbb F= \mathbb C \text{ or } \mathbb R$, $\mathbb Z_{\ge0}$ be the set of nonnegative integers, and $\mathbb N$ be the set of positive integers. 

\subsection{Lie Algebras}

\begin{definition} 
	A vector space, $L$, over $\mathbb F$, with a bilinear operator $[.,.]$ called the \textbf{Lie bracket} (this implies that if $x,y\in L$, then $[x,y]\in L$) is called a \textbf{Lie algebra} if the following axioms are satisfied for all $x,y,z\in L$:
	\begin{eqnarray}
	\left[x,x\right]&=&0\\
	\left[x,[y,z]\right]+[y,[z,x]]+[z,[x,y]]&=&0
	\end{eqnarray}
\end{definition}

This first axiom is the alternating property. The second axiom is the Jacobi identity. Through this bilinearity and the alternating property, we can derive that $[x,y]=-[y,x]$. This implies that the Lie bracket is skew-symmetric, or has the anticommutative property. 

\subsection{Nilpotency}
\begin{definition}
    The \textbf{derived Lie algebra} of the Lie algebra $L$ is the subalgebra, denoted $[L,L]$, that consists of all linear combinations of Lie brackets of pairs of elements of $L$.
\end{definition}

\begin{definition}
    The \textbf{lower central series} of the Lie algebra $L$ is the sequence of subalgebras
    $$L\supseteq[L,L]\supseteq[L,[L,L]]\supseteq[L,[L,[L,L]]]\supseteq\cdots$$
    If we write $L_1:=L$ and $L_n=[L,L_{n-1}]$ for $n\geq 2$, then the lower central series becomes
    $$L_1\supseteq L_2\supseteq L_3\supseteq\cdots$$
\end{definition}

\begin{definition}
    A Lie algebra is \textbf{nilpotent} if its lower central series eventually arrives at the zero algebra. Thus $L$ is nilpotent if $L_n=0$ for some $n\in\mathbb N$.
\end{definition}

\begin{theorem}[Engel's Theorem]\label{engels}
    A finite dimensional Lie algebra $L$ is nilpotent if 
    $$[x,[x,[\cdots[x,[x,y]]\cdots]]]=0$$
    for all $x,y\in L$. In other words, repeated bracketing of $x$ on $y$ eventually gives 0.
\end{theorem}


\subsection{The Universal Enveloping Algebra}



\begin{definition}
    Given a Lie algebra, $L$, the Universal Enveloping Algebra of $L$, $\mathbf{U}(L)$, is the unique unital associative algebra over $\mathbb F$ such that there is a linear function $\iota:L\rightarrow \mathbf{U}(L)$, such that for all $x,y\in L$
    \begin{equation}\label{commutator}
        \iota([x,y])=\iota(x)\iota(y)-\iota(y)\iota(x)
    \end{equation}
    and the following holds: for any unital associative algebra, $A$, over $\mathbb F$ and any linear function $\sigma:L\rightarrow A$ satisfying \eqref{commutator}, there exists a unique algebra homomorphism $\phi:\mathbf{U}(L)\rightarrow A$ such that $\phi(1)=1$ and $\phi\circ\iota=\sigma$
\end{definition}

Informally, $\mathbf{U}(L)$ is made up of all linear combinations over $\mathbb F$ of all formal products of the basis elements of $L$ in any order, such that for all $x,y\in L$, $xy=yx+[x,y]$ in $\mathbf{U}(L)$. This relation essentially measures how much two elements fail to commute: i.e. if $[x,y]=0$, then $xy=yx$ in $\mathbf{U}(L)$.

\begin{definition}
    Given any algebra $A$, over any field, with basis $B=\{e_i\}$ the \textbf{structure constants} of $A$ with respect to the basis $B$ are the scalars $c_{i,j}^k$ such that
    \begin{equation*}
        e_ie_j=\sum_kc_{i,j}^ke_k
    \end{equation*}
\end{definition}

Therefore, in order to obtain the structure constants, we must reorder the product of two general basis elements of $\mathbf{U}(L)$. We will find the structure constants for $\mathbf{U}(L)$, where $L$ is nilpotent, with dimension less than or equal to five.



\begin{theorem}[The  Poincar\'e-Birkhoff-Witt (PBW) Theorem]
    Let $L$ be a Lie algebra with ordered basis $\{x_1,\ldots,x_m\}$. Then the ordered products
    $$\left\{x_{1}^{r_1}\dots x_{m}^{r_m}\ |\ r_1,\ldots,r_m\in\mathbb Z_{\ge0}\right\}$$
    form a basis for $\mathbf{U}(L)$.
\end{theorem}



The PBW Theorem implies that we can reorder products of basis elements of $L$, resulting in a unique linear combination of products of these elements, in a specified order where the coefficients are the structure constants.

\subsection{Straightening Process}

\begin{definition}
    Define the Heisenberg Lie algebra, denoted $\mathfrak n_{3,1}$, to be the vector space with ordered basis $\{x_1,x_2,x_3\}$ where the only nonzero Lie bracket among the basis elements is $[x_3,x_2]=-x_1$
        
\end{definition}

By the PBW Theorem a basis for $\mathbf{U}(\mathfrak n_{3,1})$ is given by $x_1^jx_2^kx_3^l$ where $j,k,l\in\mathbb Z_{\ge0}$. Therefore, in order to find the structure constants for $\mathbf{U}(\mathfrak n_{3,1})$, we need to rewrite  the product $(x_1^jx_2^kx_3^l)(x_1^mx_2^nx_3^p)$ where $j,k,l,m,n,p\in\mathbb Z_{\ge0}$ as a linear combination of products in the PBW order $x_1^ux_2^vx_3^w$. Such a formula is called a \textbf{straightening identity}. 

Since $x_1$ commutes with everything we have,
\begin{equation*}
    (x_1^jx_2^kx_3^l)(x_1^mx_2^nx_3^p)=x_1^{j+m}x_2^kx_3^lx_2^nx_3^p=x_1^{j+m}x_2^k(x_3^lx_2^n)x_3^p
\end{equation*}
We will use the following straightening identity to straighten $x_3^lx_2^n$.

In 2024, Chamberlin, Peck, and Rafizadeh in \cite{CPR} formulated the following straightening identity within the universal enveloping algebra of the Heisenberg algebra.  

\begin{theorem}[Chamberlin, Peck, and Rafizadeh Theorem]{\label{CPR}}
    Given a Lie algebra $L$ and $a,b,c\in L$, with $[a,b]=c$, $[a,c]=0$, and $[b,c]=0$, then for all $r,s\in\mathbb Z_{\ge 0}$:
    \begin{equation*}
    a^{r}b^{s}=\sum_{j=0}^{\min\{r,s\}}\binom{r}{j}\binom{s}{j}j!c^{j}b^{s-j}a^{r-j}
    \end{equation*}
\end{theorem}
Note that the proof for Theorem \ref{CPR}, seen in Section 3 of \cite{CPR}, holds for any $a,b,c$ in a Lie algebra with the given brackets.

Now we can get the structure constants for $\mathbf{U}(\mathfrak n_{3,1})$.

\begin{theorem}\label{n3,1} For all $j,k,l,m,n,p\in\mathbb Z_{\ge0}$, we have

\begin{equation*}
    \left(x_1^jx_2^kx_3^l\right)\left(x_1^mx_2^nx_3^p\right)=\sum_{\alpha=0}^{\min\{l,n\}}\binom{l}{\alpha}\binom{n}{\alpha}\alpha!(-1)^{\alpha}x_1^{j+m+\alpha}x_2^{k+n-\alpha}x_3^{l+p-\alpha}
\end{equation*}
\end{theorem}

\begin{proof}
    \begin{eqnarray*}
    \left(x_1^jx_2^kx_3^l\right)\left(x_1^mx_2^nx_3^p\right)
    &=&x_1^{j+m}x_2^k(x_3^lx_2^n)x_3^p\\
    &=&x_1^{j+m}x_2^k\left(\sum_{\alpha=0}^{\min\{l,n\}}\binom{l}{\alpha}\binom{n}{\alpha}\alpha!(-1)^{\alpha}x_1^{\alpha}x_2^{n-\alpha}x_3^{l-\alpha}\right)x_3^p\\
    &&\text{by Theorem }\ref{CPR}\\
    &=&\sum_{\alpha=0}^{\min\{l,n\}}\binom{l}{\alpha}\binom{n}{\alpha}\alpha!(-1)^{\alpha}x_1^{j+m+\alpha}x_2^{k+n-\alpha}x_3^{l+p-\alpha}
\end{eqnarray*}
\end{proof}

For the remaining structure constant calculations, we needed a few straightening identities of our own. 

\section{Necessary Straightening Identities}

The proofs of our lemmas require the use of double induction. This technique is used to prove a property $P(r, s)$ depending on two natural numbers $r$ and $s$. It involves first proving the base case $P(1, 1)$, then proving the implication $P(r, 1) \implies P(r + 1, 1)$ for some $r\in\mathbb N$ and finally proving the implication $P(r, s) \implies P(r, s + 1)$ for some $s\in\mathbb N$ where $r$ is general in $\mathbb N$.\\

We will use divided powers to simplify the proofs. 

\begin{definition}
    Let $L$ be a Lie algebra, with $u\in\mathbf {U}(L)$, and $k\in\mathbb Z$. Define the \textbf{$kth$ divided power} of an element in $\mathbf{U}(L)$ as follows: $u^{(k)}=0$ if $k<0$ and 

    \begin{equation*}
        u^{(k)}=\frac{u^{k}}{k!}
    \end{equation*}
    for $k\ge 0$.
\end{definition}

The following divided power version of Theorem \ref{CPR} is used in our proofs.

\begin{equation}\label{CPRdiv}
    a^{(r)}b^{(s)}=\sum_{j=0}^{\min\{r,s\}}c^{(j)}b^{(s-j)}a^{(r-j)}
\end{equation}

\begin{lemma}\label{[a,b]=c, [a,c]=d}
Given a Lie algebra $L$ and $a,b,c,d\in L$ with $[a,b]=c$, $[a,c]=d$, and all other brackets between these elements are $0$, then for all  $t,u\in\mathbb Z_{\geq0}$:

  \begin{equation*}
    a^tb^u=\sum_{\substack{k_1,k_2\in\mathbb Z_{\ge0}\\k_1+k_2\leq u\\k_1+2k_2\leq t}}\binom{u}{k_1+k_2}\binom{t}{k_1+2k_2}\binom{k_1+k_2}{k_1}\frac{(k_1+2k_2)!}{2^{k_2}}d^{k_2}c^{k_1}b^{u-k_1-k_2}a^{t-k_1-2k_2}
  \end{equation*}

Using divided powers, this formula is equivalent to the following simpler equation:

  \begin{equation*}
      a^{(t)}b^{(u)}=\sum_{\substack{k_1,k_2\in\mathbb Z_{\ge0}\\k_1+k_2\leq u\\k_1+2k_2\leq t}}\frac{1}{2^{k_2}}d^{(k_2)}c^{(k_1)}b^{(u-k_1-k_2)}a^{(t-k_1-2k_2)}
  \end{equation*}

\end{lemma}


\begin{proof} 
If $t$ or $u$ is zero, the formulas are trivially true. So assume that $t,u\in\mathbb N$. We will prove the divided power formula by double induction on $t$ and $u$.\\

 Base case $t=u=1$:
    \begin{eqnarray*}
        ab&=&ba+[a,b]\\
        &=&ba+c
    \end{eqnarray*}
    
If $u=1$, the formula becomes:

    \begin{equation}\label{base3.3}
        a^{(t)}b=ba^{(t)}+ca^{(t-1)}+\frac{1}{2}da^{(t-2)}
    \end{equation}

We will prove this equation by induction on t. The base case was done above. Now assume the formula for some $t\in\mathbb N$

\begin{eqnarray*}
    (t+1)a^{(t+1)}b&=&aa^{(t)}b\\
    &=&a\left(ba^{(t)}+ca^{(t-1)}+\frac{1}{2}da^{(t-2)}\right) \text{ by the induction hypothesis}\\
    &=&aba^{(t)}+aca^{(t-1)}+\frac{1}{2}ada^{(t-2)}\\
    &=&(ba+c)a^{(t)}+(ca+d)a^{(t-1)}+\frac{1}{2}(t-1)da^{(t-1)}\\
    &=&(t+1)ba^{(t+1)}+ca^{(t)}+tca^{(t)}+da^{(t-1)}+\frac{1}{2}(t-1)da^{(t-1)}\\
    &=&(t+1)ba^{(t+1)}+(t+1)ca^{(t)}+\frac{1}{2}(t+1)da^{(t-1)}\\
    &=&(t+1)\left(ba^{(t+1)}+ca^{(t)}+\frac{1}{2}da^{(t-1)}\right)
\end{eqnarray*}

Thus \eqref{base3.3} is proven by induction on t.\\

\noindent Induction Hypothesis: Assume that the formula holds for some $u\in\mathbb N$ and all $t\in\mathbb N$, then
    \begin{eqnarray*}
        (u+1)a^{(t)}b^{(u+1)}&=&a^{(t)}bb^{(u)}\\
        &=&\left(ba^{(t)}+ca^{(t-1)}+\frac{1}{2}da^{(t-2)}\right)b^{(u)} \text{ by \eqref{base3.3}}\\
        &=&ba^{(t)}b^{(u)}+ca^{(t-1)}b^{(u)}+\frac{1}{2}da^{(t-2)}b^{(u)}\\
        &=&b\left(\sum_{\substack{k_1,k_2\in\mathbb Z_+\\k_1+k_2\leq u\\k_1+2k_2\leq t}}\frac{1}{2^{k_2}}d^{(k_2)}c^{(k_1)}b^{(u-k_1-k_2)}a^{(t-k_1-2k_2)}\right)\\
        &&+c\left(\sum_{\substack{k_1,k_2\in\mathbb Z_+\\k_1+k_2\leq u\\k_1+2k_2\leq t-1}}\frac{1}{2^{k_2}}d^{(k_2)}c^{(k_1)}b^{(u-k_1-k_2)}a^{(t-1-k_1-2k_2)}\right)\\
        &&+\frac{1}{2}d\left(\sum_{\substack{k_1,k_2\in\mathbb Z_+\\k_1+k_2\leq u\\k_1+2k_2\leq t-2}}\frac{1}{2^{k_2}}d^{(k_2)}c^{(k_1)}b^{(u-k_1-k_2)}a^{(t-2-k_1-2k_2)}\right) \text{ by the induction hypothesis}\\
        &=&\sum_{\substack{k_1,k_2\in\mathbb Z_+\\k_1+k_2\leq u\\k_1+2k_2\leq t}}\frac{1}{2^{k_2}}(u+1-k_1-k_2)d^{(k_2)}c^{(k_1)}b^{(u+1-k_1-k_2)}a^{(t-k_1-2k_2)}\\
        &&+\sum_{\substack{k_1,k_2\in\mathbb Z_+\\k_1+k_2\leq u\\k_1+2k_2\leq t-1}}\frac{1}{2^{k_2}}(k_1+1)d^{(k_2)}c^{(k_1+1)}b^{(u-k_1-k_2)}a^{(t-1-k_1-2k_2)}\\
        &&+\sum_{\substack{k_1,k_2\in\mathbb Z_+\\k_1+k_2\leq u\\k_1+2k_2\leq t-2}}\frac{1}{2^{k_2+1}}(k_2+1)d^{(k_2+1)}c^{(k_1)}b^{(u-k_1-k_2)}a^{(t-2-k_1-2k_2)}\\
        &=&\sum_{\substack{k_1,k_2\in\mathbb Z_+\\k_1+k_2\leq u+1\\k_1+2k_2\leq t}}\frac{1}{2^{k_2}}(u+1-k_1-k_2)d^{(k_2)}c^{(k_1)}b^{(u+1-k_1-k_2)}a^{(t-k_1-2k_2)}\\
        &&+\sum_{\substack{k_1,k_2\in\mathbb Z_+\\k_1+k_2\leq u+1\\k_1+2k_2\leq t}}\frac{1}{2^{k_2}}k_1d^{(k_2)}c^{(k_1)}b^{(u+1-k_1-k_2)}a^{(t-k_1-2k_2)}\\
        &&+\sum_{\substack{k_1,k_2\in\mathbb Z_+\\k_1+k_2\leq u+1\\k_1+2k_2\leq t}}\frac{1}{2^{k_2}}k_2d^{(k_2)}c^{(k_1)}b^{(u+1-k_1-k_2)}a^{(t-k_1-2k_2)}\\
        &=&(u+1)\sum_{\substack{k_1,k_2\in\mathbb Z_+\\k_1+k_2\leq u+1\\k_1+2k_2\leq t}}\frac{1}{2^{k_2}}d^{(k_2)}c^{(k_1)}b^{(u+1-k_1-k_2)}a^{(t-k_1-2k_2)}
    \end{eqnarray*}

Thus, Lemma \ref{[a,b]=c, [a,c]=d} is proven by double induction on $t$ and $u$.

\end{proof}

\begin{lemma}\label{[a,b]=c, [b,c]=d, [a,c]=g}
Given a Lie algebra $L$, and $a,b,c,d,g\in L$ and $t,u\in\mathbb Z_{\ge0}$ with $[a,b]=c$, $[b,c]=d$, and $[a,c]=g$, and all other brackets between these elements are $0$, then 
    \begin{eqnarray*} 
    a^tb^u&=&\sum_{\substack{k_1,k_2,k_3\in\mathbb Z_{\ge 0}\\k_1+2k_2+k_3\le u\\k_1+k_2+2k_3\le t}}\binom{u}{k_1+2k_2+k_3}\binom{t}{k_1+k_2+2k_3}\frac{(k_1+2k_2+k_3)!(k_1+k_2+2k_3)!}{2^{k_2+k_3}k_3!k_2!k_1!}\\
    &&\times g^{k_3}d^{k_2}c^{k_1}b^{u-k_1-2k_2-k_3}a^{t-k_1-k_2-2k_3}
\end{eqnarray*}

In divided powers, 

\begin{eqnarray*}
    a^{(t)}b^{(u)}&=&\sum_{\substack{k_1,k_2,k_3\in\mathbb Z_{\ge 0}\\k_1+2k_2+k_3\le u\\k_1+k_2+2k_3\le t}}\frac{1}{2^{k_2+k_3}}g^{(k_3)}d^{(k_2)}c^{(k_1)}b^{(u-k_1-2k_2-k_3)}a^{(t-k_1-k_2-2k_3)}
\end{eqnarray*}

\end{lemma}

\begin{proof}
If $t$ or $u$ is zero, the formulas are trivially true. So assume that $t,u\in\mathbb N$. We will prove the divided power formula by double induction on $t$ and $u$.\\

Base Case: $t=u=1$
\begin{eqnarray*}
    ab&=&ba+[a,b]\\
    &=&ba+c
\end{eqnarray*}

If $t=1$ the formula becomes:

\begin{equation}\label{base3.4}
    ab^{(u)}=b^{(u)}a+cb^{(u-1)}+\frac{1}{2}db^{(u-2)}
\end{equation}

We will prove this equation by induction on $u$. The base case was done above. Now assume this formula for some $u\in\mathbb N$

\begin{eqnarray*}
    (u+1)ab^{(u+1)}&=&ab^{(u)}b\\
    &=&\left(b^{(u)}a+cb^{(u-1)}+\frac{1}{2}db^{(u-2)}\right)b \text{ by the induction hypothesis}\\
    &=&b^{(u)}ab+cb^{(u-1)}b+\frac{1}{2}db^{(u-2)}b\\
    &=&b^{(u)}(ba+c)+ucb^{(u)}+(u-1)\frac{1}{2}db^{(u-1)}\\
    &=&(u+1)b^{(u+1)}a+b^{(u)}c+ucb^{(u)}+(u-1)\frac{1}{2}db^{(u-1)}\\
    &=&(u+1)b^{(u+1)}a+(cb^{(u)}+db^{(u-1)})+ucb^{(u)}+(u-1)\frac{1}{2}db^{(u-1)} \text{ by \eqref{CPRdiv}}\\
    &=&(u+1)b^{(u+1)}a+cb^{(u)}+db^{(u-1)}+ucb^{(u)}+(u-1)\frac{1}{2}db^{(u-1)}\\
    &=&(u+1)b^{(u+1)}a+(u+1)cb^{(u)}+(u+1)\frac{1}{2}db^{(u-1)}\\
    &=&(u+1)\left(b^{(u+1)}a+cb^{(u)}+\frac{1}{2}db^{(u-1)}\right)
\end{eqnarray*}

Thus \eqref{base3.4} is proven by induction on $u$.\\

\noindent Induction Hypothesis: Assume that the formula holds for some $t\in\mathbb N$ and all $u\in\mathbb N$, then 

\begin{eqnarray*}
    &&(t+1)a^{(t+1)}b^{(u)}=a^{(t)}ab^{(u)}\\
    &=&a^{(t)}\left(b^{(u)}a+cb^{(u-1)}+\frac{1}{2}db^{(u-2)}\right) \text{ by \eqref{base3.4}}\\
    &=&a^{(t)}b^{(u)}a+a^{(t)}cb^{(u-1)}+\frac{1}{2}a^{(t)}db^{(u-2)}\\
    &=&\left(\sum_{\substack{k_1,k_2,k_3\in\mathbb Z_{\ge 0}\\k_1+2k_2+k_3\le u\\k_1+k_2+2k_3\le t}}\frac{1}{2^{k_2+k_3}}g^{(k_3)}d^{(k_2)}c^{(k_1)}b^{(u-k_1-2k_2-k_3)}a^{(t-k_1-k_2-2k_3)}\right)a\\
    &&+(ca^{(t)}+ga^{(t-1)})b^{(u-1)} \text{ by \eqref{CPRdiv}}\\
    &&+\frac{1}{2}d\left(\sum_{\substack{k_1,k_2,k_3\in\mathbb Z_{\ge 0}\\k_1+2k_2+k_3\le u-2\\k_1+k_2+2k_3\le t}}\frac{1}{2^{k_2+k_3}}g^{(k_3)}d^{(k_2)}c^{(k_1)}b^{(u-2-k_1-2k_2-k_3)}a^{(t-k_1-k_2-2k_3)}\right)\\
    &&\text{ by the induction hypothesis}\\
    &=&\sum_{\substack{k_1,k_2,k_3\in\mathbb Z_{\ge 0}\\k_1+2k_2+k_3\le u\\k_1+k_2+2k_3\le t}}\frac{1}{2^{k_2+k_3}}(t+1-k_1-k_2-2k_3)g^{(k_3)}d^{(k_2)}c^{(k_1)}b^{(u-k_1-2k_2-k_3)}a^{(t+1-k_1-k_2-2k_3)}\\
    &&+c\left(\sum_{\substack{k_1,k_2,k_3\in\mathbb Z_{\ge 0}\\k_1+2k_2+k_3\le u-1\\k_1+k_2+2k_3\le t}}\frac{1}{2^{k_2+k_3}}g^{(k_3)}d^{(k_2)}c^{(k_1)}b^{(u-1-k_1-2k_2-k_3)}a^{(t-k_1-k_2-2k_3)}\right)\\
    &&+g\left(\sum_{\substack{k_1,k_2,k_3\in\mathbb Z_{\ge 0}\\ k_1+2k_2+k_3\le u-1\\k_1+k_2+2k_3\le t-1}}\frac{1}{2^{k_2+k_3}}g^{(k_3)}d^{(k_2)}c^{(k_1)}b^{(u-1-k_1-2k_2-k_3)}a^{(t-1-k_1-k_2-2k_3)}\right) \\
    &&\text{ by the induction hypothesis}\\
    &&+\sum_{\substack{k_1,k_2,k_3\in\mathbb Z_{\ge 0}\\k_1+2k_2+k_3\le u-2\\k_1+k_2+2k_3\le t}}\frac{1}{2^{k_2+k_3+1}}(k_2+1)g^{(k_3)}d^{(k_2+1)}c^{(k_1)}b^{(u-2-k_1-2k_2-k_3)}a^{(t-k_1-k_2-2k_3)}\\
    &=&\sum_{\substack{k_1,k_2,k_3\in\mathbb Z_{\ge 0}\\k_1+2k_2+k_3\le u\\k_1+k_2+2k_3\le t}}\frac{1}{2^{k_2+k_3}}(t+1-k_1-k_2-2k_3)g^{(k_3)}d^{(k_2)}c^{(k_1)}b^{(u-k_1-2k_2-k_3)}a^{(t+1-k_1-k_2-2k_3)}\\
    &&+\sum_{\substack{k_1,k_2,k_3\in\mathbb Z_{\ge 0}\\k_1+2k_2+k_3\le u-1\\k_1+k_2+2k_3\le t}}\frac{1}{2^{k_2+k_3}}(k_1+1)g^{(k_3)}d^{(k_2)}c^{(k_1+1)}b^{(u-1-k_1-2k_2-k_3)}a^{(t-k_1-k_2-2k_3)}\\
    &&+\sum_{\substack{k_1,k_2,k_3\in\mathbb Z_{\ge 0}\\ k_1+2k_2+k_3\le u-1\\k_1+k_2+2k_3\le t-1}}\frac{1}{2^{k_2+k_3}}(k_3+1)g^{(k_3+1)}d^{(k_2)}c^{(k_1)}b^{(u-1-k_1-2k_2-k_3)}a^{(t-1-k_1-k_2-2k_3)} \\
    &&+\sum_{\substack{k_1,k_2,k_3\in\mathbb Z_{\ge 0}\\k_1+2k_2+k_3\le u-2\\k_1+k_2+2k_3\le t}}\frac{1}{2^{k_2+k_3+1}}(k_2+1)g^{(k_3)}d^{(k_2+1)}c^{(k_1)}b^{(u-2-k_1-2k_2-k_3)}a^{(t-k_1-k_2-2k_3)}\\
    &=&\sum_{\substack{k_1,k_2,k_3\in\mathbb Z_{\ge 0}\\k_1+2k_2+k_3\le u\\k_1+k_2+2k_3\le t+1}}\frac{1}{2^{k_2+k_3}}(t+1-k_1-k_2-2k_3)g^{(k_3)}d^{(k_2)}c^{(k_1)}b^{(u-k_1-2k_2-k_3)}a^{(t+1-k_1-k_2-2k_3)}\\
    &&+\sum_{\substack{k_1,k_2,k_3\in\mathbb Z_{\ge 0}\\ k_1+2k_2+k_3\le u\\k_1+k_2+2k_3\le t+1}}\frac{1}{2^{k_2+k_3}}k_1g^{(k_3)}d^{(k_2)}c^{(k_1)}b^{(u-k_1-2k_2-k_3)}a^{(t+1-k_1-k_2-2k_3)}\\
    &&+\sum_{\substack{k_1,k_2,k_3\in\mathbb Z_{\ge 0}\\ k_1+2k_2+k_3\le u\\k_1+k_2+2k_3\le t+1}}\frac{2}{2^{k_2+k_3}}k_3g^{(k_3)}d^{(k_2)}c^{(k_1)}b^{(u-k_1-2k_2-k_3)}a^{(t+1-k_1-k_2-2k_3)} \\
    &&+\sum_{\substack{k_1,k_2,k_3\in\mathbb Z_{\ge 0}\\ k_1+2k_2+k_3\le u\\k_1+k_2+2k_3\le t+1}}\frac{1}{2^{k_2+k_3}}k_2g^{(k_3)}d^{(k_2)}c^{(k_1)}b^{(u-k_1-2k_2-k_3)}a^{(t+1-k_1-k_2-2k_3)}\\
    &=&(t+1)\sum_{\substack{k_1,k_2,k_3\in\mathbb Z_{\ge 0}\\ k_1+2k_2+k_3\le u\\k_1+k_2+2k_3\le t+1}}\frac{1}{2^{k_2+k_3}}g^{(k_3)}d^{(k_2)}c^{(k_1)}b^{(u-k_1-2k_2-k_3)}a^{(t+1-k_1-k_2-2k_3)}
\end{eqnarray*}

Thus, Lemma \ref{[a,b]=c, [b,c]=d, [a,c]=g} is proven by double induction on $t$ and $u$.

\end{proof}

\begin{lemma}\label{[a,b]=c, [a,c]=d, [a,d]=g}

Given a Lie algebra $L$ and $a,b,c,d,g\in L$, with $[a,b]=c$, $[a,c]=d$, and $[a,d]=g$, and all other brackets between these elements are $0$, then for all  $t,u\in\mathbb Z_{\geq0}$:

        \begin{eqnarray*}
        a^tb^u&=&\sum_{\substack{k_1,k_2,k_3\in\mathbb Z_{\geq0}\\k_1+k_2+k_3\leq u\\k_1+2k_2+3k_3\leq t}}\binom{u}{k_1+k_2+k_3}\binom{t}{k_1+2k_2+3k_3}\frac{(k_1+k_2+k_3)!(k_1+2k_2+3k_3)!}{(2!)^{k_2}(3!)^{k_3}k_3!k_2!k_1!}\\
        &&\times g^{k_3}d^{k_2}c^{k_1}b^{u-k_1-k_2-k_3}a^{t-k_1-2k_2-3k_3}
    \end{eqnarray*}

    In divided powers, 
    
    \begin{equation*}
        a^{(t)}b^{(u)}=\sum_{\substack{k_1,k_2,k_3\in\mathbb Z_{\geq0}\\k_1+k_2+k_3\leq u\\k_1+2k_2+3k_3\leq t}}\frac{1}{(2!)^{k_2}(3!)^{k_3}}g^{(k_3)}d^{(k_2)}c^{(k_1)}b^{(u-k_1-k_2-k_3)}a^{(t-k_1-2k_2-3k_3)}
        \end{equation*}
\end{lemma}

\begin{proof}
If $t$ or $u$ is zero, the formulas are trivially true. So assume that $t,u\in\mathbb N$. We will prove the divided power formula by double induction on $t$ and $u$.\\

Base Case: $t=u=1$
\begin{eqnarray*}
    ab&=&ba+[a,b]\\
    &=&ba+c
\end{eqnarray*}

If $u=1$ the formula becomes:

\begin{equation}\label{base3.5}
    a^{(t)}b=ba^{(t)}+ca^{(t-1)}+\frac{1}{2}da^{(t-2)}+\frac{1}{6}ga^{(t-3)}
\end{equation}

We will prove this equation by induction on $t$. The base case was done above. Now assume this formula for some $t\in\mathbb N$.

\begin{eqnarray*}
    (t+1)a^{(t+1)}b&=&aa^{(t)}b\\
    &=&a\left(ba^{(t)}+ca^{(t-1)}+\frac{1}{2}da^{(t-2)}+\frac{1}{6}ga^{(t-3)}\right) \text{ by the induction hypothesis}\\
    &=&aba^{(t)}+aca^{(t-1)}+\frac{1}{2}ada^{(t-2)}+(t-2)\frac{1}{6}ga^{(t-2)}\\
    &=&\left(ba+c\right)a^{(t)}+\left(ca+d\right)a^{(t-1)}+\frac{1}{2}\left(da+g\right)a^{(t-2)}+(t-2)\frac{1}{6}ga^{(t-2)}\\
    &=&(t+1)ba^{(t+1)}+ca^{(t)}+tca^{(t)}+da^{(t-1)}+(t-1)\frac{1}{2}da^{(t-1)}+\frac{1}{2}ga^{(t-2)}+(t-2)\frac{1}{6}ga^{(t-2)}\\
    &=&(t+1)ba^{(t+1)}+(t+1)ca^{(t)}+(t+1)\frac{1}{2}da^{(t-1)}+(t+1)\frac{1}{6}ga^{(t-2)}\\
    &=&(t+1)\left(ba^{(t+1)}+ca^{(t)}+\frac{1}{2}da^{(t-1)}+\frac{1}{6}ga^{(t-2)}\right)
\end{eqnarray*}

Thus \eqref{base3.5} is proven by induction on $t$.\\

\noindent Induction Hypothesis: Assume that the formula holds for some $u\in\mathbb N$ and all $t\in\mathbb N$, then 

\begin{eqnarray*}
    &&(u+1)a^{(t)}b^{(u+1)}=a^{(t)}bb^{(u)}\\
    &=&\big(ba^{(t)}+ca^{(t-1)}+\frac{1}{2}da^{(t-2)}+\frac{1}{6}ga^{(t-3)}\big)b^{(u)} \text{ by \eqref{base3.5}}\\
    &=&ba^{(t)}b^{(u)}+ca^{(t-1)}b^{(u)}+\frac{1}{2}da^{(t-2)}b^{(u)}+\frac{1}{6}ga^{(t-3)}b^{(u)}\\
    &=&b\left(\sum_{\substack{k_3,k_2,k_1\in\mathbb Z_{\ge 0}\\k_1+k_2+k_3\le u\\k_1+2k_2+3k_3\le t\\}}\frac{1}{(2!)^{k_2}(3!)^{k_3}}g^{(k_3)}d^{(k_2)}c^{(k_1)}b^{(u-k_1-k_2-k_3)}a^{(t-k_1-2k_2-3k_3)}\right)\\
    &&+c\left(\sum_{\substack{k_3,k_2,k_1\in\mathbb Z_{\ge 0}\\k_1+k_2+k_3\le u\\k_1+2k_2+3k_3\le t-1\\}}\frac{1}{(2!)^{k_2}(3!)^{k_3}}g^{(k_3)}d^{(k_2)}c^{(k_1)}b^{(u-k_1-k_2-k_3)}a^{(t-1-k_1-2k_2-3k_3)}\right)\\
    &&+\frac{1}{2}d\left(\sum_{\substack{k_3,k_2,k_1\in\mathbb Z_{\ge 0}\\k_1+k_2+k_3\le u\\k_1+2k_2+3k_3\le t-2\\}}\frac{1}{(2!)^{k_2}(3!)^{k_3}}g^{(k_3)}d^{(k_2)}c^{(k_1)}b^{(u-k_1-k_2-k_3)}a^{(t-2-k_1-2k_2-3k_3)}\right)\\
    &&+\frac{1}{3!}g\left(\sum_{\substack{k_3,k_2,k_1\in\mathbb Z_{\ge 0}\\k_1+k_2+k_3\le u\\k_1+2k_2+3k_3\le t-3\\}}\frac{1}{(2!)^{k_2}(3!)^{k_3}}g^{(k_3)}d^{(k_2)}c^{(k_1)}b^{(u-k_1-k_2-k_3)}a^{(t-3-k_1-2k_2-3k_3)}\right)\\
    &&\text{ by the induction hypothesis}\\
    &=&\sum_{\substack{k_3,k_2,k_1\in\mathbb Z_{\ge 0}\\k_1+k_2+k_3\le u\\k_1+2k_2+3k_3\le t\\}}\frac{1}{(2!)^{k_2}(3!)^{k_3}}(u+1-k_1-k_2-k_3)g^{(k_3)}d^{(k_2)}c^{(k_1)}b^{(u+1-k_1-k_2-k_3)}a^{(t-k_1-2k_2-3k_3)}\\
    &&+\sum_{\substack{k_3,k_2,k_1\in\mathbb Z_{\ge 0}\\k_1+k_2+k_3\le u\\k_1+2k_2+3k_3\le t-1\\}}\frac{1}{(2!)^{k_2}(3!)^{k_3}}(k_1+1)g^{(k_3)}d^{(k_2)}c^{(k_1+1)}b^{(u-k_1-k_2-k_3)}a^{(t-1-k_1-2k_2-3k_3)}\\
    &&+\sum_{\substack{k_3,k_2,k_1\in\mathbb Z_{\ge 0}\\k_1+k_2+k_3\le u\\k_1+2k_2+3k_3\le t-2\\}}\frac{1}{(2!)^{k_2+1}(3!)^{k_3}}(k_2+1)g^{(k_3)}d^{(k_2+1)}c^{(k_1)}b^{(u-k_1-k_2-k_3)}a^{(t-2-k_1-2k_2-3k_3)}\\
    &&+\sum_{\substack{k_3,k_2,k_1\in\mathbb Z_{\ge 0}\\k_1+k_2+k_3\le u\\k_1+2k_2+3k_3\le t-3\\}}\frac{1}{(2!)^{k_2}(3!)^{k_3+1}}(k_3+1)g^{(k_3+1)}d^{(k_2)}c^{(k_1)}b^{(u-k_1-k_2-k_3)}a^{(t-3-k_1-2k_2-3k_3)}\\
    &=&\sum_{\substack{k_3,k_2,k_1\in\mathbb Z_{\ge 0}\\k_1+k_2+k_3\le u+1\\k_1+2k_2+3k_3\le t\\}}\frac{1}{(2!)^{k_2}(3!)^{k_3}}(u+1-k_1-k_2-k_3)g^{(k_3)}d^{(k_2)}c^{(k_1)}b^{(u+1-k_1-k_2-k_3)}a^{(t-k_1-2k_2-3k_3)}\\
    &&+\sum_{\substack{k_3,k_2,k_1\in\mathbb Z_{\ge 0}\\k_1+k_2+k_3\le u+1\\k_1+2k_2+3k_3\le t\\}}\frac{1}{(2!)^{k_2}(3!)^{k_3}}k_1g^{(k_3)}d^{(k_2)}c^{(k_1)}b^{(u+1-k_1-k_2-k_3)}a^{(t-k_1-2k_2-3k_3)}\\
    &&+\sum_{\substack{k_3,k_2,k_1\in\mathbb Z_{\ge 0}\\k_1+k_2+k_3\le u+1\\k_1+2k_2+3k_3\le t\\}}\frac{1}{(2!)^{k_2}(3!)^{k_3}}k_2g^{(k_3)}d^{(k_2)}c^{(k_1)}b^{(u+1-k_1-k_2-k_3)}a^{(t-k_1-2k_2-3k_3)}\\
    &&+\sum_{\substack{k_3,k_2,k_1\in\mathbb Z_{\ge 0}\\k_1+k_2+k_3\le u+1\\k_1+2k_2+3k_3\le t\\}}\frac{1}{(2!)^{k_2}(3!)^{k_3}}k_3g^{(k_3)}d^{(k_2)}c^{(k_1)}b^{(u+1-k_1-k_2-k_3)}a^{(t-k_1-2k_2-3k_3)}\\
    &=&(u+1)\sum_{\substack{k_3,k_2,k_1\in\mathbb Z_{\ge 0}\\k_1+k_2+k_3\le u+1\\k_1+2k_2+3k_3\le t\\}}\frac{1}{(2!)^{k_2}(3!)^{k_3}}g^{(k_3)}d^{(k_2)}c^{(k_1)}b^{(u+1-k_1-k_2-k_3)}a^{(t-k_1-2k_2-3k_3)}
\end{eqnarray*}

Thus, Lemma \ref{[a,b]=c, [a,c]=d, [a,d]=g} is proven by double induction on $t$ and $u$.

\end{proof}

\begin{lemma}\label{[a,b]=c, [a,c]=d, [a,d]=g, [b,c]=-g}

Given a Lie algebra $L$, $a,b,c,d,g\in\mathbb Z_{\ge0}$ with $[a,b]=c$, $[a,c]=d$, $[a,d]=g$, and
 $[b,c]=-g$,and all other brackets between these elements is $0$, then \\ 

\begin{eqnarray*}
    a^{t}b^{u}&=&\sum_{\substack{k_1,k_2,k_3,k_4\in\mathbb Z_{\ge0}\\k_1+k_2+k_3+2k_4\le u\\k_1+2k_2+3k_3+k_4\le t}}\binom{u}{k_1+k_2+k_3+2k_4}\binom{t}{k_1+2k_2+3k_3+k_4}\\
    &&\times \frac{(k_1+k_2+k_3+2k_4)!(k_1+2k_2+3k_3+k_4)!}{(2!)^{k_2+k_4}(3!)^{k_3}k_4!k_3!k_2!k_1!}(-g)^{k_4}g^{k_3}d^{k_2}c^{k_1}b^{u-k_1-k_2-k_3-2k_4}a^{t-k_1-2k_2-3k_3-k_4}
\end{eqnarray*}

In divided powers, 

 \begin{eqnarray*}
    a^{(t)}b^{(u)}&=&\sum_{\substack{k_1,k_2,k_3,k_4\in\mathbb Z_{\ge0}\\k_1+k_2+k_3+2k_4\le u\\k_1+2k_2+3k_3+k_4\le t}}\frac{1}{(2!)^{k_2+k_4}(3!)^{k_3}}(-g)^{(k_4)}g^{(k_3)}d^{(k_2)}c^{(k_1)}b^{(u-k_1-k_2-k_3-2k_4)}a^{(t-k_1-2k_2-3k_3-k_4)}
\end{eqnarray*}

\end{lemma}

\begin{proof} If $t$ or $u$ is zero, the formulas are trivially true. So assume that $t,u\in\mathbb N$. We will prove the divided power formula by double induction on $t$ and $u$.\\

Base Case: $t=u=1$
\begin{eqnarray*}
    ab&=&ba+[a,b]\\
    &=&ba+c
\end{eqnarray*}

If $t=1$ the formula becomes:

\begin{equation}\label{base3.6}
    ab^{(u)}=b^{(u)}a+cb^{(u-1)}-\frac{1}{2}gb^{(u-2)}
\end{equation}

We will prove this equation by induction on $u$. The base case was done above. Now assume this formula for some $u\in\mathbb N$.

\begin{eqnarray*}
    (u+1)ab^{(u+1)}&=&ab^{(u)}b\\
    &=&\left(b^{(u)}a+cb^{(u-1)}-\frac{1}{2}gb^{(u-2)}\right)b\text{ by the induction hypothesis}\\
    &=&b^{(u)}ab+cb^{(u-1)}b-\frac{1}{2}gb^{(u-2)}b\\
    &=&b^{(u)}(ba+c)+ucb^{(u)}-\frac{1}{2}(u-1)gb^{(u-1)}\\
    &=&b^{(u)}ba+b^{(u)}c+ucb^{(u)}-\frac{1}{2}(u-1)gb^{(u-1)}\\
    &=&(u+1)b^{(u+1)}a+(cb^{(u)}-gb^{(u-1)})+ucb^{(u)}-\frac{1}{2}(u-1)gb^{(u-1)} \text{ by Theorem \ref{CPR}}\\
    &=&(u+1)b^{(u+1)}a+cb^{(u)}-gb^{(u-1)}+ucb^{(u)}-\frac{1}{2}(u-1)gb^{(u-1)}\\
    &=&(u+1)b^{(u+1)}a+(u+1)cb^{(u)}-\frac{1}{2}(u+1)gb^{(u-1)}
\end{eqnarray*}

Thus \eqref{base3.6} is proven by induction on $u$.\\

\noindent Induction Hypothesis: Assume the formula holds for some $t\in\mathbb N$ and all $u\in\mathbb N$, then

\begin{eqnarray*}
    &&(t+1)a^{(t+1)}b^{(u)}=a^{(t)}ab^{(u)}\\
    &=&a^{(t)}\left(b^{(u)}a+cb^{(u-1)}-\frac{1}{2}gb^{(u-2)}\right) \text{ by \eqref{base3.6}}\\
    &=&a^{(t)}b^{(u)}a+a^{(t)}cb^{(u-1)}-\frac{1}{2}ga^{(t)}b^{(u-2)} \\
    &=&a^{(t)}b^{(u)}a+\left(ca^{(t)}+da^{(t-1)}+\frac{1}{2}ga^{(t-2)}\right)b^{(u-1)}-\frac{1}{2}ga^{(t)}b^{(u-2)} \text{ by \eqref{base3.3}}\\
    &=&a^{(t)}b^{(u)}a+ca^{(t)}b^{(u-1)}+da^{(t-1)}b^{(u-1)}+\frac{1}{2}ga^{(t-2)}b^{(u-1)}-\frac{1}{2}ga^{(t)}b^{(u-2)}\\
    &=&\left(\sum_{\substack{k_1,k_2,k_3,k_4\in\mathbb Z_{\ge0}\\k_1+k_2+k_3+2k_4\le u\\k_1+2k_2+3k_3+k_4\le t}}\frac{1}{(2!)^{k_2+k_4}(3!)^{k_3}}(-g)^{(k_4)}g^{(k_3)}d^{(k_2)}c^{(k_1)}b^{(u-k_1-k_2-k_3-2k_4)}a^{(t-k_1-2k_2-3k_3-k_4)}\right)a\\
    &&+c\left(\sum_{\substack{k_1,k_2,k_3,k_4\in\mathbb Z_{\ge0}\\k_1+k_2+k_3+2k_4\le u-1\\k_1+2k_2+3k_3+k_4\le t}}\frac{1}{(2!)^{k_2+k_4}(3!)^{k_3}}(-g)^{(k_4)}g^{(k_3)}d^{(k_2)}c^{(k_1)}b^{(u-1-k_1-k_2-k_3-2k_4)}a^{(t-k_1-2k_2-3k_3-k_4)}\right)\\
    &&+d\left(\sum_{\substack{k_1,k_2,k_3,k_4\in\mathbb Z_{\ge0}\\k_1+k_2+k_3+2k_4\le u-1\\k_1+2k_2+3k_3+k_4\le t-1}}\frac{1}{(2!)^{k_2+k_4}(3!)^{k_3}}(-g)^{(k_4)}g^{(k_3)}d^{(k_2)}c^{(k_1)}b^{(u-1-k_1-k_2-k_3-2k_4)}a^{(t-1-k_1-2k_2-3k_3-k_4)}\right)\\
    &&+\frac{1}{2}g\left(\sum_{\substack{k_1,k_2,k_3,k_4\in\mathbb Z_{\ge0}\\k_1+k_2+k_3+2k_4\le u-1\\k_1+2k_2+3k_3+k_4\le t-2}}\frac{1}{(2!)^{k_2+k_4}(3!)^{k_3}}(-g)^{(k_4)}g^{(k_3)}d^{(k_2)}c^{(k_1)}b^{(u-1-k_1-k_2-k_3-2k_4)}a^{(t-2-k_1-2k_2-3k_3-k_4)}\right)\\
    &&-\frac{1}{2}g\left(\sum_{\substack{k_1,k_2,k_3,k_4\in\mathbb Z_{\ge0}\\k_1+k_2+k_3+2k_4\le u-2\\k_1+2k_2+3k_3+k_4\le t}}\frac{1}{(2!)^{k_2+k_4}(3!)^{k_3}}(-g)^{(k_4)}g^{(k_3)}d^{(k_2)}c^{(k_1)}b^{(u-2-k_1-k_2-k_3-2k_4)}a^{(t-k_1-2k_2-3k_3-k_4)}\right)\\
    &&\text{ by the induction hypothesis}\\
    &=&\sum_{\substack{k_1,k_2,k_3,k_4\in\mathbb Z_{\ge0}\\k_1+k_2+k_3+2k_4\le u\\k_1+2k_2+3k_3+k_4\le t}}\frac{1}{(2!)^{k_2+k_4}(3!)^{k_3}}(t+1-k_1-2k_2-3k_3-k_4)(-g)^{(k_4)}g^{(k_3)}d^{(k_2)}c^{(k_1)}b^{(u-k_1-k_2-k_3-2k_4)}\\
    &&\times a^{(t+1-k_1-2k_2-3k_3-k_4)}\\
    &&+\sum_{\substack{k_1,k_2,k_3,k_4\in\mathbb Z_{\ge0}\\k_1+k_2+k_3+2k_4\le u-1\\k_1+2k_2+3k_3+k_4\le t}}\frac{1}{(2!)^{k_2+k_4}(3!)^{k_3}}(k_1+1)(-g)^{(k_4)}g^{(k_3)}d^{(k_2)}c^{(k_1+1)}b^{(u-1-k_1-k_2-k_3-2k_4)}\\
    &&\times a^{(t-k_1-2k_2-3k_3-k_4)}\\
    &&+\sum_{\substack{k_1,k_2,k_3,k_4\in\mathbb Z_{\ge0}\\k_1+k_2+k_3+2k_4\le u-1\\k_1+2k_2+3k_3+k_4\le t-1}}\frac{1}{(2!)^{k_2+k_4}(3!)^{k_3}}(k_2+1)(-g)^{(k_4)}g^{(k_3)}d^{(k_2+1)}c^{(k_1)}b^{(u-1-k_1-k_2-k_3-2k_4)}\\
    &&\times a^{(t-1-k_1-2k_2-3k_3-k_4)}\\
    &&+\sum_{\substack{k_1,k_2,k_3,k_4\in\mathbb Z_{\ge0}\\k_1+k_2+k_3+2k_4\le u-1\\k_1+2k_2+3k_3+k_4\le t-2}}\frac{1}{(2!)^{k_2+k_4+1}(3!)^{k_3}}(k_3+1)(-g)^{(k_4)}g^{(k_3+1)}d^{(k_2)}c^{(k_1)}b^{(u-1-k_1-k_2-k_3-2k_4)}\\
    &&\times a^{(t-2-k_1-2k_2-3k_3-k_4)}\\
    &&+\sum_{\substack{k_1,k_2,k_3,k_4\in\mathbb Z_{\ge0}\\k_1+k_2+k_3+2k_4\le u-2\\k_1+2k_2+3k_3+k_4\le t}}\frac{1}{(2!)^{k_2+k_4+1}(3!)^{k_3}}(k_4+1)(-g)^{(k_4+1)}g^{(k_3)}d^{(k_2)}c^{(k_1)}b^{(u-2-k_1-k_2-k_3-2k_4)}\\
    &&\times a^{(t-k_1-2k_2-3k_3-k_4)}\\
    &=&\sum_{\substack{k_1,k_2,k_3,k_4\in\mathbb Z_{\ge0}\\k_1+k_2+k_3+2k_4\le u\\k_1+2k_2+3k_3+k_4\le t+1}}\frac{1}{(2!)^{k_2+k_4}(3!)^{k_3}}(t+1-k_1-2k_2-3k_3-k_4)(-g)^{(k_4)}g^{(k_3)}d^{(k_2)}c^{(k_1)}b^{(u-k_1-k_2-k_3-2k_4)}\\
    &&\times a^{(t+1-k_1-2k_2-3k_3-k_4)}\\
    &&+\sum_{\substack{k_1,k_2,k_3,k_4\in\mathbb Z_{\ge0}\\k_1+k_2+k_3+2k_4\le u\\k_1+2k_2+3k_3+k_4\le t+1}}\frac{1}{(2!)^{k_2+k_4}(3!)^{k_3}}k_1(-g)^{(k_4)}g^{(k_3)}d^{(k_2)}c^{(k_1)}b^{(u-k_1-k_2-k_3-2k_4)}a^{(t+1-k_1-2k_2-3k_3-k_4)}\\
    &&+\sum_{\substack{k_1,k_2,k_3,k_4\in\mathbb Z_{\ge0}\\k_1+k_2+k_3+2k_4\le u\\k_1+2k_2+3k_3+k_4\le t+1}}\frac{2}{(2!)^{k_2+k_4}(3!)^{k_3}}k_2(-g)^{(k_4)}g^{(k_3)}d^{(k_2)}c^{(k_1)}b^{(u-k_1-k_2-k_3-2k_4)}a^{(t+1-k_1-2k_2-3k_3-k_4)}\\
    &&+\sum_{\substack{k_1,k_2,k_3,k_4\in\mathbb Z_{\ge0}\\k_1+k_2+k_3+2k_4\le u\\k_1+2k_2+3k_3+k_4\le t+1}}\frac{3}{(2!)^{k_2+k_4}(3!)^{k_3}}k_3(-g)^{(k_4)}g^{(k_3)}d^{(k_2)}c^{(k_1)}b^{(u-k_1-k_2-k_3-2k_4)}a^{(t+1-k_1-2k_2-3k_3-k_4)}\\
    &&+\sum_{\substack{k_1,k_2,k_3,k_4\in\mathbb Z_{\ge0}\\k_1+k_2+k_3+2k_4\le u\\k_1+2k_2+3k_3+k_4\le t+1}}\frac{1}{(2!)^{k_2+k_4}(3!)^{k_3}}k_4(-g)^{(k_4)}g^{(k_3)}d^{(k_2)}c^{(k_1)}b^{(u-k_1-k_2-k_3-2k_4)}a^{(t+1-k_1-2k_2-3k_3-k_4)}\\
    &=&(t+1)\sum_{\substack{k_1,k_2,k_3,k_4\in\mathbb Z_{\ge0}\\k_1+k_2+k_3+2k_4\le u\\k_1+2k_2+3k_3+k_4\le t+1}}\frac{1}{(2!)^{k_2+k_4}(3!)^{k_3}}(-g)^{(k_4)}g^{(k_3)}d^{(k_2)}c^{(k_1)}b^{(u-k_1-k_2-k_3-2k_4)}a^{(t+1-k_1-2k_2-3k_3-k_4)}
\end{eqnarray*}

Thus, Lemma \ref{[a,b]=c, [a,c]=d, [a,d]=g, [b,c]=-g} is proven by double induction on $t$ and $u$.

\end{proof}

We can now derive the remaining structure constant formulas. 

\section{Structure Constant Formulas}

We will now derive the structure constant formulas for the nilpotent Lie algebras of dimension four and five, according to \v{S}nobl and Winternitz's classification, see \cite{SW}. 

\begin{definition}
    Define the Lie algebra $\mathfrak n_{4,1}$ to be the vector space with ordered basis $\{x_1,x_2,x_3,x_4\}$, where the only nonzero brackets of basis elements are $[x_4,x_2]=-x_1$ and $[x_4,x_3]=-x_2$. 
        
\end{definition}

\begin{theorem}
    For all $j,k,l,m,n,p,q,r\in\mathbb Z_{\ge0}$, we have $(x_1^jx_2^kx_3^lx_4^m)(x_1^nx_2^px_3^qx_4^r)$ 

\begin{eqnarray*}
    &=&\sum_{\alpha=0}^{\min\{m,p\}}\sum_{\substack{\beta_1,\beta_2\in\mathbb Z_{\ge0}\\\beta_1+\beta_2\leq q\\\beta_1+2\beta_2\leq m-\alpha}}\binom{m}{\alpha}\binom{p}{\alpha}\binom{q}{\beta_1+\beta_2}\binom{m-\alpha}{\beta_1+2\beta_2}\binom{\beta_1+\beta_2}{\beta_1}\alpha!\frac{(\beta_1+2\beta_2)!}{2^{\beta_2}}(-1)^{\alpha+\beta_1}\\
    &&\times x_1^{j+n+\alpha+\beta_2}x_2^{k+p-\alpha+\beta_1}x_3^{l+q-\beta_1-\beta_2}x_4^{m+r-\alpha-\beta_1-2\beta_2}
\end{eqnarray*}
\end{theorem}

\begin{proof}
\begin{eqnarray*}
    &&(x_1^jx_2^kx_3^lx_4^m)(x_1^nx_2^px_3^qx_4^r)\\
    &=&x_1^{j+n}x_2^kx_3^l(x_4^mx_2^p)x_3^qx_4^r\\
    &=&x_1^{j+n}x_2^kx_3^l\left(\sum_{\alpha=0}^{\min\{m,p\}}\binom{m}{\alpha}\binom{p}{\alpha}\alpha!(-1)^{\alpha}x_1^{\alpha}x_2^{p-\alpha}x_4^{m-\alpha}\right)x_3^qx_4^r \text{ by Theorem \ref{CPR}}\\
    &=&\sum_{\alpha=0}^{\min\{m,p\}}\binom{m}{\alpha}\binom{p}{\alpha}\alpha!(-1)^{\alpha}x_1^{j+n+\alpha}x_2^{k+p-\alpha}x_3^l\left(x_4^{m-\alpha}x_3^q\right)x_4^r\\
    &=&\sum_{\alpha=0}^{\min\{m,p\}}\binom{m}{\alpha}\binom{p}{\alpha}\alpha!(-1)^{\alpha}x_1^{j+n+\alpha}x_2^{k+p-\alpha}x_3^l\Bigg(\sum_{\substack{\beta_1,\beta_2\in\mathbb Z_{\ge0}\\\beta_1+\beta_2\leq q\\\beta_1+2\beta_2\leq m-\alpha}}\binom{q}{\beta_1+\beta_2}\binom{m-\alpha}{\beta_1+2\beta_2}\binom{\beta_1+\beta_2}{\beta_1}\frac{(\beta_1+2\beta_2)!}{2^{\beta_2}}\\
    &&\times (-1)^{\beta_1}x_1^{\beta_2}x_2^{\beta_1}x_3^{q-\beta_1-\beta_2}x_4^{m-\alpha-\beta_1-2\beta_2}\Bigg)x_4^r \text{ by Lemma \ref{[a,b]=c, [a,c]=d}}\\
    &=&\sum_{\alpha=0}^{\min\{m,p\}}\sum_{\substack{\beta_1,\beta_2\in\mathbb Z_{\ge0}\\\beta_1+\beta_2\leq q\\\beta_1+2\beta_2\leq m-\alpha}}\binom{m}{\alpha}\binom{p}{\alpha}\binom{q}{\beta_1+\beta_2}\binom{m-\alpha}{\beta_1+2\beta_2}\binom{\beta_1+\beta_2}{\beta_1}\alpha!\frac{(\beta_1+2\beta_2)!}{2^{\beta_2}}(-1)^{\alpha+\beta_1}\\
    &&\times x_1^{j+n+\alpha+\beta_2}x_2^{k+p-\alpha+\beta_1}x_3^{l+q-\beta_1-\beta_2}x_4^{m+r-\alpha-\beta_1-2\beta_2}
\end{eqnarray*}
\end{proof}

\begin{definition}
    Define the Lie algebra $\mathfrak n_{5,1}$ to be the vector space with ordered basis $\{x_1,x_2,x_3,x_4,x_5\}$, where the only nonzero brackets of basis elements are $[x_5,x_3]=-x_1$ and $[x_5,x_4]=-x_2$.
        
\end{definition}

\begin{theorem}\label{n5,1} For all $j,k,l,m,n,p,q,r,s,t\in\mathbb Z_{\ge0}$, we have: $(x_1^jx_2^kx_3^lx_4^mx_5^n)(x_1^px_2^qx_3^rx_4^sx_5^t)$ 
    \begin{eqnarray*}
        &=&\sum_{\alpha=0}^{\min\{n,r\}}\sum_{\beta=0}^{\min\{n-\alpha,s\}}\binom{n}{\alpha}\binom{r}{\alpha}\binom{n-\alpha}{\beta}\binom{s}{\beta}\alpha!\beta!(-1)^{\alpha+\beta} x_1^{j+p+\alpha}x_2^{k+q+\beta}x_3^{l+r-\alpha}x_4^{m+s-\beta}x_5^{n+t-\alpha-\beta}
    \end{eqnarray*}
\end{theorem}
\begin{proof}
\begin{eqnarray*}
    &&(x_1^jx_2^kx_3^lx_4^mx_5^n)(x_1^px_2^qx_3^rx_4^sx_5^t)\\
    &=&x_1^{j+p}x_2^{k+q}x_3^lx_4^m(x_5^nx_3^r)x_4^sx_5^t\\
    &=&x_1^{j+p}x_2^{k+q}x_3^lx_4^m\left(\sum_{\alpha=0}^{\min\{n,r\}}\binom{n}{\alpha}\binom{r}{\alpha}\alpha!(-1)^\alpha x_1^{\alpha}x_3^{r-\alpha}x_5^{n-\alpha}\right)x_4^sx_5^t \text{ by Theorem \ref{CPR}}\\
    &=&\sum_{\alpha=0}^{\min\{n,r\}}\binom{n}{\alpha}\binom{r}{\alpha}\alpha!(-1)^\alpha x_1^{j+p+\alpha}x_2^{k+q}x_3^{l+r-\alpha}x_4^m(x_5^{n-\alpha}x_4^s)x_5^t\\
    &=&\sum_{\alpha=0}^{\min\{n,r\}}\binom{n}{\alpha}\binom{r}{\alpha}\alpha!(-1)^\alpha x_1^{j+p+\alpha}x_2^{k+q}x_3^{l+r-\alpha}x_4^m\left(\sum_{\beta=0}^{\min\{n-\alpha,s\}}\binom{n-\alpha}{\beta}\binom{s}{\beta}\beta!(-1)^{\beta}x_2^{\beta}x_4^{s-\beta}x_5^{n-\alpha-\beta}\right)x_5^t \\
    &&\text{ by Theorem \ref{CPR}}\\
    &=&\sum_{\alpha=0}^{\min\{n,r\}}\sum_{\beta=0}^{\min\{n-\alpha,s\}}\binom{n}{\alpha}\binom{r}{\alpha}\binom{n-\alpha}{\beta}\binom{s}{\beta}\alpha!\beta!(-1)^{\alpha+\beta} x_1^{j+p+\alpha}x_2^{k+q+\beta}x_3^{l+r-\alpha}x_4^{m+s-\beta}x_5^{n+t-\alpha-\beta}
\end{eqnarray*}
\end{proof}

\begin{definition}
    Define the Lie algebra $\mathfrak n_{5,2}$ to be the vector space with ordered basis $\{x_1,x_2,x_3,x_4,x_5\}$ where the only nonzero brackets of basis elements are $[x_4,x_3]=-x_2$, $[x_5,x_3]=-x_1$, and $[x_5,x_4]=-x_3$.
        
\end{definition}

\begin{theorem}\label{n5,2} For all $j,k,l,m,n,p,q,r,s,t\in\mathbb Z_{\ge0}$, we have: $(x_1^jx_2^kx_3^lx_4^mx_5^n)(x_1^px_2^qx_3^rx_4^sx_5^t)$ 
\begin{eqnarray*}
    &=&\sum_{\alpha=0}^{\min\{n,r\}}\sum_{\substack{\beta_1,\beta_2,\beta_3\in\mathbb Z_{\ge 0}\\\beta_1+2\beta_2+\beta_3\le s\\\beta_1+\beta_2+2\beta_3\le n-\alpha}}\sum_{\gamma=0}^{\min\{m,r-\alpha+\beta_1\}}\binom{n}{\alpha}\binom{r}{\alpha}\binom{s}{\beta_1+2\beta_2+\beta_3}\binom{n-\alpha}{\beta_1+\beta_2+2\beta_3}\binom{m}{\gamma}\binom{r-\alpha+\beta_1}{\gamma}\\
    &&\times \frac{(\beta_1+2\beta_2+\beta_3)!(\beta_1+\beta_2+2\beta_3)!}{2^{\beta_2+\beta_3}\beta_3!\beta_2!\beta_1!}(-1)^{\alpha+\beta_1+\gamma}\alpha!\gamma!x_1^{j+p+\alpha+\beta_3}x_2^{k+q+\beta_2+\gamma}x_3^{l+r-\alpha+\beta_1-\gamma}\\
    &&\times x_4^{m+s-\gamma-\beta_1-2\beta_2-\beta_3}x_5^{n+t-\alpha-\beta_1-\beta_2-2\beta_3} 
\end{eqnarray*}
\end{theorem}

\begin{proof}
\begin{eqnarray*}
    &&(x_1^jx_2^kx_3^lx_4^mx_5^n)(x_1^px_2^qx_3^rx_4^sx_5^t)\\
    &=&x_1^{j+p}x_2^{k+q}x_3^lx_4^m(x_5^nx_3^r)x_4^sx_5^t\\
    &=&x_1^{j+p}x_2^{k+q}x_3^lx_4^m\left(\sum_{\alpha=0}^{\min\{n,r\}}\binom{n}{\alpha}\binom{r}{\alpha}(-1)^{\alpha}\alpha!x_1^{\alpha}x_3^{r-\alpha}x_5^{n-\alpha}\right)x_4^sx_5^t \text{ by Theorem \ref{CPR}}\\
    &=&\sum_{\alpha=0}^{\min\{n,r\}}\binom{n}{\alpha}\binom{r}{\alpha}(-1)^{\alpha}\alpha!x_1^{j+p+\alpha}x_2^{k+q}x_3^lx_4^mx_3^{r-\alpha}(x_5^{n-\alpha}x_4^s)x_5^t\\
    &=&\sum_{\alpha=0}^{\min\{n,r\}}\binom{n}{\alpha}\binom{r}{\alpha}(-1)^{\alpha}\alpha!x_1^{j+p+\alpha}x_2^{k+q}x_3^lx_4^mx_3^{r-\alpha}\bigg(\sum_{\substack{\beta_1,\beta_2,\beta_3\in\mathbb Z_{\ge 0}\\\beta_1+2\beta_2+\beta_3\le s\\\beta_1+\beta_2+2\beta_3\le n-\alpha}}\binom{s}{\beta_1+2\beta_2+\beta_3}\binom{n-\alpha}{\beta_1+\beta_2+2\beta_3}\\
    &&\times \frac{(\beta_1+2\beta_2+\beta_3)!(\beta_1+\beta_2+2\beta_3)!}{2^{\beta_2+\beta_3}\beta_3!\beta_2!\beta_1!}(-1)^{\beta_1}x_1^{\beta_3}x_2^{\beta_2}x_3^{\beta_1}x_4^{s-\beta_1-2\beta_2-\beta_3}x_5^{n-\alpha-\beta_1-\beta_2-2\beta_3}\bigg)x_5^t \text{ by Lemma \ref{[a,b]=c, [b,c]=d, [a,c]=g}}\\
    &=&\sum_{\alpha=0}^{\min\{n,r\}}\sum_{\substack{\beta_1,\beta_2,\beta_3\in\mathbb Z_{\ge 0}\\\beta_1+2\beta_2+\beta_3\le s\\\beta_1+\beta_2+2\beta_3\le n-\alpha}}\binom{n}{\alpha}\binom{r}{\alpha}\binom{s}{\beta_1+2\beta_2+\beta_3}\binom{n-\alpha}{\beta_1+\beta_2+2\beta_3}\frac{(\beta_1+2\beta_2+\beta_3)!(\beta_1+\beta_2+2\beta_3)!}{2^{\beta_2+\beta_3}\beta_3!\beta_2!\beta_1!}\\
    &&\times (-1)^{\alpha+\beta_1}\alpha!x_1^{j+p+\alpha+\beta_3}x_2^{k+q+\beta_2}x_3^l(x_4^mx_3^{r-\alpha+\beta_1})x_4^{s-\beta_1-2\beta_2-\beta_3}x_5^{n+t-\alpha-\beta_1-\beta_2-2\beta_3}\\
    &=&\sum_{\alpha=0}^{\min\{n,r\}}\sum_{\substack{\beta_1,\beta_2,\beta_3\in\mathbb Z_{\ge 0}\\\beta_1+2\beta_2+\beta_3\le s\\\beta_1+\beta_2+2\beta_3\le n-\alpha}}\binom{n}{\alpha}\binom{r}{\alpha}\binom{s}{\beta_1+2\beta_2+\beta_3}\binom{n-\alpha}{\beta_1+\beta_2+2\beta_3}\frac{(\beta_1+2\beta_2+\beta_3)!(\beta_1+\beta_2+2\beta_3)!}{2^{\beta_2+\beta_3}\beta_3!\beta_2!\beta_1!}\\
    &&\times (-1)^{\alpha+\beta_1}\alpha!x_1^{j+p+\alpha+\beta_3}x_2^{k+q+\beta_2}x_3^l\left(\sum_{\gamma=0}^{\min\{m,r-\alpha+\beta_1\}}\binom{m}{\gamma}\binom{r-\alpha+\beta_1}{\gamma}(-1)^{\gamma}\gamma!x_2^{\gamma}x_3^{r-\alpha+\beta_1-\gamma}x_4^{m-\gamma}\right)\\
    &&\times x_4^{s-\beta_1-2\beta_2-\beta_3}x_5^{n+t-\alpha-\beta_1-\beta_2-2\beta_3} \text{ by Theorem \ref{CPR}}\\
    &=&\sum_{\alpha=0}^{\min\{n,r\}}\sum_{\substack{\beta_1,\beta_2,\beta_3\in\mathbb Z_{\ge 0}\\\beta_1+2\beta_2+\beta_3\le s\\\beta_1+\beta_2+2\beta_3\le n-\alpha}}\sum_{\gamma=0}^{\min\{m,r-\alpha+\beta_1\}}\binom{n}{\alpha}\binom{r}{\alpha}\binom{s}{\beta_1+2\beta_2+\beta_3}\binom{n-\alpha}{\beta_1+\beta_2+2\beta_3}\binom{m}{\gamma}\binom{r-\alpha+\beta_1}{\gamma}\\
    &&\times \frac{(\beta_1+2\beta_2+\beta_3)!(\beta_1+\beta_2+2\beta_3)!}{2^{\beta_2+\beta_3}\beta_3!\beta_2!\beta_1!}(-1)^{\alpha+\beta_1+\gamma}\alpha!\gamma!x_1^{j+p+\alpha+\beta_3}x_2^{k+q+\beta_2+\gamma}x_3^{l+r-\alpha+\beta_1-\gamma}\\
    &&\times x_4^{m+s-\gamma-\beta_1-2\beta_2-\beta_3}x_5^{n+t-\alpha-\beta_1-\beta_2-2\beta_3} 
\end{eqnarray*}
\end{proof}

\begin{definition}
    Define the Lie algebra $\mathfrak n_{5,3}$ to be the vector space with ordered basis $\{x_1,x_2,x_3,x_4,x_5\}$, where the only nonzero brackets of basis elements are $[x_4,x_2]=-x_1$ and $[x_5,x_3]=-x_1$.
        
\end{definition}

\begin{theorem}\label{n5,3} For all $j,k,l,m,n,p,q,r,s,t\in\mathbb Z_{\ge0}$, we have: $(x_1^jx_2^kx_3^lx_4^mx_5^n)(x_1^px_2^qx_3^rx_4^sx_5^t)$:

\begin{eqnarray*}
    &=&\sum_{\alpha=0}^{\min\{n,r\}}\sum_{\beta=0}^{\min\{m,q\}}\binom{n}{\alpha}\binom{r}{\alpha}\binom{m}{\beta}\binom{q}{\beta}\alpha!\beta!(-1)^{\alpha+\beta} x_1^{j+p+\alpha+\beta}x_2^{k+q-\beta}x_3^{l+r-\alpha} x_4^{m+s-\beta}x_5^{n+t-\alpha}
\end{eqnarray*}
\end{theorem}
\begin{proof}
\begin{eqnarray*}
    &&(x_1^jx_2^kx_3^lx_4^mx_5^n)(x_1^px_2^qx_3^rx_4^sx_5^t)\\
    &=&x_1^{j+p}x_2^kx_3^lx_4^mx_2^q(x_5^nx_3^r)x_4^sx_5^t\\
    &=&x_1^{j+p}x_2^kx_3^lx_4^mx_2^q\left(\sum_{\alpha=0}^{\min\{n,r\}}\binom{n}{\alpha}\binom{r}{\alpha}\alpha!(-1)^\alpha x_1^\alpha x_3^{r-\alpha}x_5^{n-\alpha}\right)x_4^sx_5^t \text{ by Theorem \ref{CPR}}\\
     &=&\sum_{\alpha=0}^{\min\{n,r\}}\binom{n}{\alpha}\binom{r}{\alpha}\alpha!(-1)^\alpha x_1^{j+p+\alpha}x_2^kx_3^{l+r-\alpha}(x_4^mx_2^q)  x_4^sx_5^{n+t-\alpha}\\
     &=&\sum_{\alpha=0}^{\min\{n,r\}}\binom{n}{\alpha}\binom{r}{\alpha}\alpha!(-1)^\alpha x_1^{j+p+\alpha}x_2^kx_3^{l+r-\alpha}\left(\sum_{\beta=0}^{\min\{m,q\}}\binom{m}{\beta}\binom{q}{\beta}\beta!(-1)^\beta x_1^\beta x_2^{q-\beta}x_4^{m-\beta}\right)x_4^sx_5^{n+t-\alpha} \\
     &&\text{ by Theorem \ref{CPR}}\\
     &=&\sum_{\alpha=0}^{\min\{n,r\}}\sum_{\beta=0}^{\min\{m,q\}}\binom{n}{\alpha}\binom{r}{\alpha}\binom{m}{\beta}\binom{q}{\beta}\alpha!\beta!(-1)^{\alpha+\beta} x_1^{j+p+\alpha+\beta}x_2^{k+q-\beta}x_3^{l+r-\alpha} x_4^{m+s-\beta}x_5^{n+t-\alpha} 
\end{eqnarray*}
\end{proof}

\begin{definition}
    Define the Lie algebra $\mathfrak n_{5,4}$ to be the vector space with ordered basis $\{x_1,x_2,x_3,x_4,x_5\}$ where the only nonzero brackets of basis elements are $[x_4,x_3]=-x_1$, $[x_5,x_2]=-x_1$, and $[x_5,x_4]=-x_2$.
        
\end{definition}

\begin{theorem}\label{n5,4} For all $j,k,l,m,n,p,q,r,s,t\in\mathbb Z_{\ge0}$, we have: $(x_1^jx_2^kx_3^lx_4^mx_5^n)(x_1^px_2^qx_3^rx_4^sx_5^t)$ 

\begin{eqnarray*}
    &=&\sum_{\alpha=0}^{\min\{n,q\}}\sum_{\substack{\beta_1,\beta_2\in\mathbb Z_{\ge0}\\\beta_1+\beta_2\leq s\\\beta_1+2\beta_2\leq n-\alpha}}\sum_{\gamma=0}^{\min\{m,r\}}\binom{n}{\alpha}\binom{q}{\alpha}\binom{s}{\beta_1+\beta_2}\binom{n-\alpha}{\beta_1+2\beta_2}\binom{\beta_1+\beta_2}{\beta_1}\binom{m}{\gamma}\binom{r}{\gamma}\alpha!\frac{(\beta_1+2\beta_2)!}{2^{\beta_2}}\gamma!\\
    &&\times (-1)^{\alpha+\beta_1+\gamma} x_1^{j+p+\alpha+\beta_2+\gamma}x_2^{k+q-\alpha+\beta_1}x_3^{l+r-\gamma}x_4^{m+s-\gamma-\beta_1-\beta_2}x_5^{n+t-\alpha-\beta_1-2\beta_2}
\end{eqnarray*}
\end{theorem}

\begin{proof}
    \begin{eqnarray*}
     &&(x_1^jx_2^kx_3^lx_4^mx_5^n)(x_1^px_2^qx_3^rx_4^sx_5^t)\\
     &=&x_1^{j+p}x_2^kx_3^lx_4^m(x_5^nx_2^q)x_3^rx_4^sx_5^t\\
     &=&x_1^{j+p}x_2^kx_3^lx_4^m\left(\sum_{\alpha=0}^{\min\{n,q\}}\binom{n}{\alpha}\binom{q}{\alpha}\alpha!(-1)^\alpha x_1^\alpha x_2^{q-\alpha}x_5^{n-\alpha}\right)x_3^rx_4^sx_5^t\text{ by Theorem \ref{CPR}}\\
     &=&\sum_{\alpha=0}^{\min\{n,q\}}\binom{n}{\alpha}\binom{q}{\alpha}\alpha!(-1)^\alpha x_1^{j+p+\alpha}x_2^{k+q-\alpha}x_3^lx_4^mx_3^r(x_5^{n-\alpha}x_4^s)x_5^t\\
     &=&\sum_{\alpha=0}^{\min\{n,q\}}\binom{n}{\alpha}\binom{q}{\alpha}\alpha!(-1)^\alpha x_1^{j+p+\alpha}x_2^{k+q-\alpha}x_3^lx_4^mx_3^r\\
     &&\times \left(\sum_{\substack{\beta_1,\beta_2\in\mathbb Z_{\ge0}\\\beta_1+\beta_2\leq s\\\beta_1+2\beta_2\leq n-\alpha}}\binom{s}{\beta_1+\beta_2}\binom{n-\alpha}{\beta_1+2\beta_2}\binom{\beta_1+\beta_2}{\beta_1}\frac{(\beta_1+2\beta_2)!}{2^{\beta_2}}(-1)^{\beta_1}x_1^{\beta_2}x_2^{\beta_1}x_4^{s-\beta_1-\beta_2}x_5^{n-\alpha-\beta_1-2\beta_2}\right)x_5^t\\
     && \text{ by Lemma \ref{[a,b]=c, [a,c]=d}}\\
     &=&\sum_{\alpha=0}^{\min\{n,q\}}\sum_{\substack{\beta_1,\beta_2\in\mathbb Z_{\ge0}\\\beta_1+\beta_2\leq s\\\beta_1+2\beta_2\leq n-\alpha}}\binom{n}{\alpha}\binom{q}{\alpha}\binom{s}{\beta_1+\beta_2}\binom{n-\alpha}{\beta_1+2\beta_2}\binom{\beta_1+\beta_2}{\beta_1}\alpha!\frac{(\beta_1+2\beta_2)!}{2^{\beta_2}}(-1)^{\alpha+\beta_1} \\
     &&\times x_1^{j+p+\alpha+\beta_2}x_2^{k+q-\alpha+\beta_1}x_3^l(x_4^mx_3^r)x_4^{s-\beta_1-\beta_2}x_5^{n+t-\alpha-\beta_1-2\beta_2}\\
     &=&\sum_{\alpha=0}^{\min\{n,q\}}\sum_{\substack{\beta_1,\beta_2\in\mathbb Z_{\ge0}\\\beta_1+\beta_2\leq s\\\beta_1+2\beta_2\leq n-\alpha}}\binom{n}{\alpha}\binom{q}{\alpha}\binom{s}{\beta_1+\beta_2}\binom{n-\alpha}{\beta_1+2\beta_2}\binom{\beta_1+\beta_2}{\beta_1}\alpha!\frac{(\beta_1+2\beta_2)!}{2^{\beta_2}}(-1)^{\alpha+\beta_1} \\
     &&\times x_1^{j+p+\alpha+\beta_2}x_2^{k+q-\alpha+\beta_1}x_3^l\left(\sum_{\gamma=0}^{\min\{m,r\}}\binom{m}{\gamma}\binom{r}{\gamma}\gamma!(-1)^\gamma x_1^\gamma x_3^{r-\gamma}x_4^{m-\gamma}\right)x_4^{s-\beta_1-\beta_2}x_5^{n+t-\alpha-\beta_1-2\beta_2}\\
     &&\text{ by Theorem \ref{CPR}}\\
     &=&\sum_{\alpha=0}^{\min\{n,q\}}\sum_{\substack{\beta_1,\beta_2\in\mathbb Z_{\ge0}\\\beta_1+\beta_2\leq s\\\beta_1+2\beta_2\leq n-\alpha}}\sum_{\gamma=0}^{\min\{m,r\}}\binom{n}{\alpha}\binom{q}{\alpha}\binom{s}{\beta_1+\beta_2}\binom{n-\alpha}{\beta_1+2\beta_2}\binom{\beta_1+\beta_2}{\beta_1}\binom{m}{\gamma}\binom{r}{\gamma}\alpha!\frac{(\beta_1+2\beta_2)!}{2^{\beta_2}}\gamma!\\
     &&\times (-1)^{\alpha+\beta_1+\gamma} x_1^{j+p+\alpha+\beta_2+\gamma}x_2^{k+q-\alpha+\beta_1}x_3^{l+r-\gamma}x_4^{m+s-\gamma-\beta_1-\beta_2}x_5^{n+t-\alpha-\beta_1-2\beta_2}
\end{eqnarray*}
\end{proof}

\begin{definition}
    Define the Lie algebra $\mathfrak n_{5,5}$ to be the vector space with ordered basis $\{x_1,x_2,x_3,x_4,x_5\}$ where the only nonzero brackets of basis elements are $[x_5,x_2]=-x_1$, $[x_5,x_3]=-x_2$, and $[x_5,x_4]=-x_3$.
        
\end{definition}

\begin{theorem}\label{n5,5} For all $j,k,l,m,n,p,q,r,s,t\in\mathbb Z_{\ge0}$, we have: $(x_1^jx_2^kx_3^lx_4^mx_5^n)(x_1^px_2^qx_3^rx_4^sx_5^t)$ 

\begin{eqnarray*}
    &=&\sum_{\alpha=0}^{\min\{n,q\}}\sum_{\substack{\beta_1,\beta_2\in\mathbb Z_{\ge0}\\\beta_1+\beta_2\leq r\\\beta_1+2\beta_2\leq n-\alpha}}\sum_{\substack{\gamma_1,\gamma_2,\gamma_3\in\mathbb Z_{\geq0}\\\gamma_1+\gamma_2+\gamma_3\leq s\\\gamma_1+2\gamma_2+3\gamma_3\leq n-\alpha-\beta_1-2\beta_2}}\binom{n}{\alpha}\binom{q}{\alpha}\binom{r}{\beta_1+\beta_2}\binom{n-\alpha}{\beta_1+2\beta_2}\binom{\beta_1+\beta_2}{\beta_1}\\
    &&\times \binom{s}{\gamma_1+\gamma_2+\gamma_3}\binom{n-\alpha-\beta_1-2\beta_2}{\gamma_1+2\gamma_2+3\gamma_3}\alpha!\frac{(\beta_1+2\beta_2)!}{2^{\beta_2}}\frac{(\gamma_1+\gamma_2+\gamma_3)!(\gamma_1+2\gamma_2+3\gamma_3)!}{(2!)^{\gamma_2}(3!)^{\gamma_3}\gamma_3!\gamma_2!\gamma_1!}\\
    &&\times (-1)^{\alpha+\beta_1+\gamma_1+\gamma_3}x_1^{j+p+\alpha+\beta_2+\gamma_3}x_2^{k+q-\alpha+\beta_1+\gamma_2}x_3^{l+r-\beta_1-\beta_2+\gamma_1}x_4^{m+s-\gamma_1-\gamma_2-\gamma_3}x_5^{n+t-\alpha-\beta_1-2\beta_2-\gamma_1-2\gamma_2-3\gamma_3}
\end{eqnarray*}
\end{theorem}
\begin{proof}
    \begin{eqnarray*}
     &&(x_1^jx_2^kx_3^lx_4^mx_5^n)(x_1^px_2^qx_3^rx_4^sx_5^t)\\
     &=&x_1^{j+p}x_2^kx_3^lx_4^m(x_5^nx_2^q)x_3^rx_4^sx_5^t\\
     &=&x_1^{j+p}x_2^kx_3^lx_4^m\left(\sum_{\alpha=0}^{\min\{n,q\}}\binom{n}{\alpha}\binom{q}{\alpha}\alpha!(-1)^\alpha x_1^\alpha x_2^{q-\alpha}x_5^{n-\alpha}\right)x_3^rx_4^sx_5^t \text{ by Theorem \ref{CPR}}\\
     &=&\sum_{\alpha=0}^{\min\{n,q\}}\binom{n}{\alpha}\binom{q}{\alpha}\alpha!(-1)^\alpha x_1^{j+p+\alpha}x_2^{k+q-\alpha}x_3^lx_4^m (x_5^{n-\alpha} x_3^r)x_4^sx_5^t\\
     &=&\sum_{\alpha=0}^{\min\{n,q\}}\binom{n}{\alpha}\binom{q}{\alpha}\alpha!(-1)^\alpha x_1^{j+p+\alpha}x_2^{k+q-\alpha}x_3^lx_4^m\\
     &&\times \left(\sum_{\substack{\beta_1,\beta_2\in\mathbb Z_{\ge0}\\\beta_1+\beta_2\leq r\\\beta_1+2\beta_2\leq n-\alpha}}\binom{r}{\beta_1+\beta_2}\binom{n-\alpha}{\beta_1+2\beta_2}\binom{\beta_1+\beta_2}{\beta_1}\frac{(\beta_1+2\beta_2)!}{2^{\beta_2}}(-1)^{\beta_1}x_1^{\beta_2}x_2^{\beta_1}x_3^{r-\beta_1-\beta_2}x_5^{n-\alpha-\beta_1-2\beta_2}\right)\\
    &&\times x_4^sx_5^t \text{ by Lemma \ref{[a,b]=c, [a,c]=d}}\\
    &=&\sum_{\alpha=0}^{\min\{n,q\}}\sum_{\substack{\beta_1,\beta_2\in\mathbb Z_{\ge0}\\\beta_1+\beta_2\leq r\\\beta_1+2\beta_2\leq n-\alpha}}\binom{n}{\alpha}\binom{q}{\alpha}\binom{r}{\beta_1+\beta_2}\binom{n-\alpha}{\beta_1+2\beta_2}\binom{\beta_1+\beta_2}{\beta_1}\alpha!\frac{(\beta_1+2\beta_2)!}{2^{\beta_2}}(-1)^{\alpha+\beta_1} \\
    &&\times x_1^{j+p+\alpha+\beta_2}x_2^{k+q-\alpha+\beta_1}x_3^{l+r-\beta_1-\beta_2}x_4^m(x_5^{n-\alpha-\beta_1-2\beta_2}x_4^s)x_5^t\\
    &=&\sum_{\alpha=0}^{\min\{n,q\}}\sum_{\substack{\beta_1,\beta_2\in\mathbb Z_{\ge0}\\\beta_1+\beta_2\leq r\\\beta_1+2\beta_2\leq n-\alpha}}\binom{n}{\alpha}\binom{q}{\alpha}\binom{r}{\beta_1+\beta_2}\binom{n-\alpha}{\beta_1+2\beta_2}\binom{\beta_1+\beta_2}{\beta_1}\alpha!\frac{(\beta_1+2\beta_2)!}{2^{\beta_2}}(-1)^{\alpha+\beta_1} \\
    &&\times x_1^{j+p+\alpha+\beta_2}x_2^{k+q-\alpha+\beta_1}x_3^{l+r-\beta_1-\beta_2}x_4^m\bigg(\sum_{\substack{\gamma_1,\gamma_2,\gamma_3\in\mathbb Z_{\geq0}\\\gamma_1+\gamma_2+\gamma_3\leq s\\\gamma_1+2\gamma_2+3\gamma_3\leq n-\alpha-\beta_1-2\beta_2}}\binom{s}{\gamma_1+\gamma_2+\gamma_3}\binom{n-\alpha-\beta_1-2\beta_2}{\gamma_1+2\gamma_2+3\gamma_3}\\
    &&\times \frac{(\gamma_1+\gamma_2+\gamma_3)!(\gamma_1+2\gamma_2+3\gamma_3)!}{(2!)^{\gamma_2}(3!)^{\gamma_3}\gamma_3!\gamma_2!\gamma_1!}(-1)^{\gamma_1+\gamma_3}x_1^{\gamma_3}x_2^{\gamma_2}x_3^{\gamma_1}x_4^{s-\gamma_1-\gamma_2-\gamma_3}x_5^{n-\alpha-\beta_1-2\beta_2-\gamma_1-2\gamma_2-3\gamma_3}\bigg)x_5^t \\
    &&\text{ by Lemma \ref{[a,b]=c, [a,c]=d, [a,d]=g}}\\
    &=&\sum_{\alpha=0}^{\min\{n,q\}}\sum_{\substack{\beta_1,\beta_2\in\mathbb Z_{\ge0}\\\beta_1+\beta_2\leq r\\\beta_1+2\beta_2\leq n-\alpha}}\sum_{\substack{\gamma_1,\gamma_2,\gamma_3\in\mathbb Z_{\geq0}\\\gamma_1+\gamma_2+\gamma_3\leq s\\\gamma_1+2\gamma_2+3\gamma_3\leq n-\alpha-\beta_1-2\beta_2}}\binom{n}{\alpha}\binom{q}{\alpha}\binom{r}{\beta_1+\beta_2}\binom{n-\alpha}{\beta_1+2\beta_2}\binom{\beta_1+\beta_2}{\beta_1}\\
    &&\times \binom{s}{\gamma_1+\gamma_2+\gamma_3}\binom{n-\alpha-\beta_1-2\beta_2}{\gamma_1+2\gamma_2+3\gamma_3}\alpha!\frac{(\beta_1+2\beta_2)!}{2^{\beta_2}}\frac{(\gamma_1+\gamma_2+\gamma_3)!(\gamma_1+2\gamma_2+3\gamma_3)!}{(2!)^{\gamma_2}(3!)^{\gamma_3}\gamma_3!\gamma_2!\gamma_1!}\\
    &&\times (-1)^{\alpha+\beta_1+\gamma_1+\gamma_3}x_1^{j+p+\alpha+\beta_2+\gamma_3}x_2^{k+q-\alpha+\beta_1+\gamma_2}x_3^{l+r-\beta_1-\beta_2+\gamma_1}x_4^{m+s-\gamma_1-\gamma_2-\gamma_3}x_5^{n+t-\alpha-\beta_1-2\beta_2-\gamma_1-2\gamma_2-3\gamma_3}
\end{eqnarray*}
\end{proof}

\begin{definition}
    Define the Lie algebra $\mathfrak n_{5,6}$ to be the vector space with ordered basis $\{x_1,x_2,x_3,x_4,x_5\}$, where the only nonzero brackets of basis elements are $[x_4,x_3]=-x_1$, $[x_5,x_2]=-x_1$, $[x_5,x_3]=-x_2$, and $[x_5,x_4]=-x_3$.
        
\end{definition}

\begin{theorem}\label{n5,6} For all $j,k,l,m,n,p,q,r,s,t\in\mathbb Z_{\ge0}$, we have: $(x_1^jx_2^kx_3^lx_4^mx_5^n)(x_1^px_2^qx_3^rx_4^sx_5^t)$

\begin{eqnarray*}
    &=&\sum_{\alpha=0}^{\min\{n,q\}}\sum_{\substack{\beta_1,\beta_2\in\mathbb Z_{\ge0}\\\beta_1+\beta_2\leq r\\\beta_1+2\beta_2\leq n-\alpha}}\sum_{\substack{\gamma_1,\gamma_2,\gamma_3,\gamma_4\in\mathbb Z_{\ge0}\\\gamma_1+\gamma_2+\gamma_3+2\gamma_4\le s\\\gamma_1+2\gamma_2+3\gamma_3+\gamma_4\le n-\alpha-\beta_1-2\beta_2}}\sum_{\delta=0}^{\min\{m,r-\beta_1-\beta_2+\gamma_1\}}\binom{n}{\alpha}\binom{q}{\alpha}\binom{r}{\beta_1+\beta_2}\binom{n-\alpha}{\beta_1+2\beta_2}\\
    &&\times \binom{\beta_1+\beta_2}{\beta_1}\binom{s}{\gamma_1+\gamma_2+\gamma_3+2\gamma_4}\binom{n-\alpha-\beta_1-2\beta_2}{\gamma_1+2\gamma_2+3\gamma_3+\gamma_4}\binom{m}{\delta}\binom{r-\beta_1-\beta_2+\gamma_1}{\delta}\alpha!\frac{(\beta_1+2\beta_2)!}{2^{\beta_2}}\\
    &&\times \frac{(\gamma_1+\gamma_2+\gamma_3+2\gamma_4)!(\gamma_1+2\gamma_2+3\gamma_3+\gamma_4)!}{(2!)^{\gamma_2+\gamma_4}(3!)^{\gamma_3}\gamma_4!\gamma_3!\gamma_2!\gamma_1!}\delta!(-1)^{\alpha+\beta_1+\gamma_1+\gamma_3+\delta}\\
    &&\times x_1^{j+p+\alpha+\beta_2+\gamma_4+\gamma_3+\delta}x_2^{k+q-\alpha+\beta_1+\gamma_2}x_3^{l+r-\beta_1-\beta_2+\gamma_1-\delta}x_4^{m+s-\delta-\gamma_1-\gamma_2-\gamma_3-2\gamma_4}x_5^{n+t-\alpha-\beta_1-2\beta_2-\gamma_1-2\gamma_2-3\gamma_3-\gamma_4} 
\end{eqnarray*}
\end{theorem}
\begin{proof}
    \begin{eqnarray*}
     &&(x_1^jx_2^kx_3^lx_4^mx_5^n)(x_1^px_2^qx_3^rx_4^sx_5^t)\\
     &=&x_1^{j+p}x_2^kx_3^lx_4^m(x_5^nx_2^q)x_3^rx_4^sx_5^t\\
     &=&x_1^{j+p}x_2^kx_3^lx_4^m\left(\sum_{\alpha=0}^{\min\{n,q\}}\binom{n}{\alpha}\binom{q}{\alpha}\alpha!(-1)^{\alpha}x_1^{\alpha}x_2^{q-\alpha}x_5^{n-\alpha}\right)x_3^rx_4^sx_5^t\text{ by Theorem \ref{CPR}}\\
     &=&\sum_{\alpha=0}^{\min\{n,q\}}\binom{n}{\alpha}\binom{q}{\alpha}\alpha!(-1)^{\alpha}x_1^{j+p+\alpha}x_2^{k+q-\alpha}x_3^lx_4^m(x_5^{n-\alpha}x_3^r)x_4^sx_5^t\\
     &=&\sum_{\alpha=0}^{\min\{n,q\}}\binom{n}{\alpha}\binom{q}{\alpha}\alpha!(-1)^{\alpha}x_1^{j+p+\alpha}x_2^{k+q-\alpha}x_3^lx_4^m\Bigg(\sum_{\substack{\beta_1,\beta_2\in\mathbb Z_{\ge0}\\\beta_1+\beta_2\leq r\\\beta_1+2\beta_2\leq n-\alpha}}\binom{r}{\beta_1+\beta_2}\binom{n-\alpha}{\beta_1+2\beta_2}\binom{\beta_1+\beta_2}{\beta_1}\\
     &&\times \frac{(\beta_1+2\beta_2)!}{2^{\beta_2}}(-1)^{\beta_1}x_1^{\beta_2}x_2^{\beta_1}x_3^{r-\beta_1-\beta_2}x_5^{n-\alpha-\beta_1-2\beta_2}\Bigg)x_4^sx_5^t\text{ by Lemma \ref{[a,b]=c, [a,c]=d}}\\
     &=&\sum_{\alpha=0}^{\min\{n,q\}}\sum_{\substack{\beta_1,\beta_2\in\mathbb Z_{\ge0}\\\beta_1+\beta_2\leq r\\\beta_1+2\beta_2\leq n-\alpha}}\binom{n}{\alpha}\binom{q}{\alpha}\binom{r}{\beta_1+\beta_2}\binom{n-\alpha}{\beta_1+2\beta_2}\binom{\beta_1+\beta_2}{\beta_1}\alpha!\frac{(\beta_1+2\beta_2)!}{2^{\beta_2}}(-1)^{\alpha+\beta_1}\\
     &&\times x_1^{j+p+\alpha+\beta_2}x_2^{k+q-\alpha+\beta_1}x_3^lx_4^mx_3^{r-\beta_1-\beta_2}\left(x_5^{n-\alpha-\beta_1-2\beta_2} x_4^s\right)x_5^t\\
     &=&\sum_{\alpha=0}^{\min\{n,q\}}\sum_{\substack{\beta_1,\beta_2\in\mathbb Z_{\ge0}\\\beta_1+\beta_2\leq r\\\beta_1+2\beta_2\leq n-\alpha}}\binom{n}{\alpha}\binom{q}{\alpha}\binom{r}{\beta_1+\beta_2}\binom{n-\alpha}{\beta_1+2\beta_2}\binom{\beta_1+\beta_2}{\beta_1}\alpha!\frac{(\beta_1+2\beta_2)!}{2^{\beta_2}}(-1)^{\alpha+\beta_1}\\
     &&\times x_1^{j+p+\alpha+\beta_2}x_2^{k+q-\alpha+\beta_1}x_3^lx_4^mx_3^{r-\beta_1-\beta_2}\Bigg(\sum_{\substack{\gamma_1,\gamma_2,\gamma_3,\gamma_4\in\mathbb Z_{\ge0}\\\gamma_1+\gamma_2+\gamma_3+2\gamma_4\le s\\\gamma_1+2\gamma_2+3\gamma_3+\gamma_4\le n-\alpha-\beta_1-2\beta_2}}\binom{s}{\gamma_1+\gamma_2+\gamma_3+2\gamma_4}\\
     &&\times \binom{n-\alpha-\beta_1-2\beta_2}{\gamma_1+2\gamma_2+3\gamma_3+\gamma_4}\frac{(\gamma_1+\gamma_2+\gamma_3+2\gamma_4)!(\gamma_1+2\gamma_2+3\gamma_3+\gamma_4)!}{(2!)^{\gamma_2+\gamma_4}(3!)^{\gamma_3}\gamma_4!\gamma_3!\gamma_2!\gamma_1!}\\
    &&\times (-1)^{\gamma_1+\gamma_3}x_1^{\gamma_4+\gamma_3}x_2^{\gamma_2}x_3^{\gamma_1}x_4^{s-\gamma_1-\gamma_2-\gamma_3-2\gamma_4}x_5^{n-\alpha-\beta_1-2\beta_2-\gamma_1-2\gamma_2-3\gamma_3-\gamma_4}\Bigg)x_5^t \text{ by Lemma \ref{[a,b]=c, [a,c]=d, [a,d]=g, [b,c]=-g}}\\
    &=&\sum_{\alpha=0}^{\min\{n,q\}}\sum_{\substack{\beta_1,\beta_2\in\mathbb Z_{\ge0}\\\beta_1+\beta_2\leq r\\\beta_1+2\beta_2\leq n-\alpha}}\sum_{\substack{\gamma_1,\gamma_2,\gamma_3,\gamma_4\in\mathbb Z_{\ge0}\\\gamma_1+\gamma_2+\gamma_3+2\gamma_4\le s\\\gamma_1+2\gamma_2+3\gamma_3+\gamma_4\le n-\alpha-\beta_1-2\beta_2}}\binom{n}{\alpha}\binom{q}{\alpha}\binom{r}{\beta_1+\beta_2}\binom{n-\alpha}{\beta_1+2\beta_2}\binom{\beta_1+\beta_2}{\beta_1}\\
    &&\times \binom{s}{\gamma_1+\gamma_2+\gamma_3+2\gamma_4}\binom{n-\alpha-\beta_1-2\beta_2}{\gamma_1+2\gamma_2+3\gamma_3+\gamma_4}\alpha!\frac{(\beta_1+2\beta_2)!}{2^{\beta_2}}\\
    &&\times \frac{(\gamma_1+\gamma_2+\gamma_3+2\gamma_4)!(\gamma_1+2\gamma_2+3\gamma_3+\gamma_4)!}{(2!)^{\gamma_2+\gamma_4}(3!)^{\gamma_3}\gamma_4!\gamma_3!\gamma_2!\gamma_1!}(-1)^{\alpha+\beta_1+\gamma_1+\gamma_3}x_1^{j+p+\alpha+\beta_2+\gamma_4+\gamma_3}x_2^{k+q-\alpha+\beta_1+\gamma_2}x_3^l\\
    &&\times (x_4^mx_3^{r-\beta_1-\beta_2+\gamma_1})x_4^{s-\gamma_1-\gamma_2-\gamma_3-2\gamma_4}x_5^{n+t-\alpha-\beta_1-2\beta_2-\gamma_1-2\gamma_2-3\gamma_3-\gamma_4} \\
    &=&\sum_{\alpha=0}^{\min\{n,q\}}\sum_{\substack{\beta_1,\beta_2\in\mathbb Z_{\ge0}\\\beta_1+\beta_2\leq r\\\beta_1+2\beta_2\leq n-\alpha}}\sum_{\substack{\gamma_1,\gamma_2,\gamma_3,\gamma_4\in\mathbb Z_{\ge0}\\\gamma_1+\gamma_2+\gamma_3+2\gamma_4\le s\\\gamma_1+2\gamma_2+3\gamma_3+\gamma_4\le n-\alpha-\beta_1-2\beta_2}}\binom{n}{\alpha}\binom{q}{\alpha}\binom{r}{\beta_1+\beta_2}\binom{n-\alpha}{\beta_1+2\beta_2}\binom{\beta_1+\beta_2}{\beta_1}\\
    &&\times \binom{s}{\gamma_1+\gamma_2+\gamma_3+2\gamma_4}\binom{n-\alpha-\beta_1-2\beta_2}{\gamma_1+2\gamma_2+3\gamma_3+\gamma_4}\alpha!\frac{(\beta_1+2\beta_2)!}{2^{\beta_2}}\\
    &&\times \frac{(\gamma_1+\gamma_2+\gamma_3+2\gamma_4)!(\gamma_1+2\gamma_2+3\gamma_3+\gamma_4)!}{(2!)^{\gamma_2+\gamma_4}(3!)^{\gamma_3}\gamma_4!\gamma_3!\gamma_2!\gamma_1!}(-1)^{\alpha+\beta_1+\gamma_1+\gamma_3}x_1^{j+p+\alpha+\beta_2+\gamma_4+\gamma_3}x_2^{k+q-\alpha+\beta_1+\gamma_2}x_3^l\\
    &&\times \left(\sum_{\delta=0}^{\min\{m,r-\beta_1-\beta_2+\gamma_1\}}\binom{m}{\delta}\binom{r-\beta_1-\beta_2+\gamma_1}{\delta}\delta!(-1)^{\delta}x_1^{\delta}x_3^{r-\beta_1-\beta_2+\gamma_1-\delta}x_4^{m-\delta}\right)\\
    &&\times x_4^{s-\gamma_1-\gamma_2-\gamma_3-2\gamma_4}x_5^{n+t-\alpha-\beta_1-2\beta_2-\gamma_1-2\gamma_2-3\gamma_3-\gamma_4} \text{ by Theorem \ref{CPR}}\\
    &=&\sum_{\alpha=0}^{\min\{n,q\}}\sum_{\substack{\beta_1,\beta_2\in\mathbb Z_{\ge0}\\\beta_1+\beta_2\leq r\\\beta_1+2\beta_2\leq n-\alpha}}\sum_{\substack{\gamma_1,\gamma_2,\gamma_3,\gamma_4\in\mathbb Z_{\ge0}\\\gamma_1+\gamma_2+\gamma_3+2\gamma_4\le s\\\gamma_1+2\gamma_2+3\gamma_3+\gamma_4\le n-\alpha-\beta_1-2\beta_2}}\sum_{\delta=0}^{\min\{m,r-\beta_1-\beta_2+\gamma_1\}}\binom{n}{\alpha}\binom{q}{\alpha}\binom{r}{\beta_1+\beta_2}\binom{n-\alpha}{\beta_1+2\beta_2}\\
    &&\times \binom{\beta_1+\beta_2}{\beta_1}\binom{s}{\gamma_1+\gamma_2+\gamma_3+2\gamma_4}\binom{n-\alpha-\beta_1-2\beta_2}{\gamma_1+2\gamma_2+3\gamma_3+\gamma_4}\binom{m}{\delta}\binom{r-\beta_1-\beta_2+\gamma_1}{\delta}\alpha!\frac{(\beta_1+2\beta_2)!}{2^{\beta_2}}\\
    &&\times \frac{(\gamma_1+\gamma_2+\gamma_3+2\gamma_4)!(\gamma_1+2\gamma_2+3\gamma_3+\gamma_4)!}{(2!)^{\gamma_2+\gamma_4}(3!)^{\gamma_3}\gamma_4!\gamma_3!\gamma_2!\gamma_1!}\delta!(-1)^{\alpha+\beta_1+\gamma_1+\gamma_3+\delta}\\
    &&\times x_1^{j+p+\alpha+\beta_2+\gamma_4+\gamma_3+\delta}x_2^{k+q-\alpha+\beta_1+\gamma_2}x_3^{l+r-\beta_1-\beta_2+\gamma_1-\delta}x_4^{m+s-\delta-\gamma_1-\gamma_2-\gamma_3-2\gamma_4}x_5^{n+t-\alpha-\beta_1-2\beta_2-\gamma_1-2\gamma_2-3\gamma_3-\gamma_4} 
\end{eqnarray*}
\end{proof}

\bibliographystyle{plain}
\bibliography{references}

{\footnotesize  
\medskip
\medskip
\vspace*{1mm} 
 
\noindent {\it Samuel Chamberlin}\\  
Park University\\
8700 River Park Drive\\
Parlville, MO\\
E-mail: {\tt schamberlin@park.edu}\\ \\  

\noindent {\it Emmerson Taylor}\\  
Park University \\
8700 River Park Drive\\
Parkville, MO \\
E-mail: {\tt 1710985@park.edu}\\ \\




\end{document}